\documentclass[12pt]{amsart}
\usepackage{amssymb}
\usepackage{mathtools}
\usepackage{graphicx}
\usepackage[cmtip,all]{xy}
\usepackage{bbm}
\usepackage{comment}
\usepackage{mathrsfs}
\usepackage{microtype}
\usepackage[T1]{fontenc}
\usepackage{xcolor}
\usepackage[dvipdfmx]{hyperref}
\usepackage{autobreak}
\hypersetup{
 colorlinks,
 linkcolor={blue!50!black},
 citecolor={blue!50!black},
 urlcolor={blue!80!black}
}

\calclayout

\def\DefineSymbol#1#2{\newcommand{#1}{{\mathrm {#2}}}}
\def\DefineCategory#1#2{\newcommand{#1}{{\mathrm {#2}}}}

\theoremstyle{plain}
	\newtheorem{theorem}{Theorem}[section]
	\newtheorem{lemma}[theorem]{Lemma}
	\newtheorem{proposition}[theorem]{Proposition}
	\newtheorem{corollary}[theorem]{Corollary}
\theoremstyle{definition}
	\newtheorem{definition}[theorem]{Definition}
	\newtheorem{lemma and definition}[theorem]{Lemma and Definition}
	\newtheorem{example}[theorem]{Example}

	\newtheorem{notation}[theorem]{Notation}
\theoremstyle{remark}
	\newtheorem{remark}[theorem]{Remark}
	\numberwithin{equation}{section}

\DefineSymbol{\pr}{pr}
\DefineSymbol{\id}{id}
\DefineSymbol{\const}{const}
\DefineSymbol{\op}{op}
\DefineSymbol{\Op}{Op}
\DefineSymbol{\diag}{diag}
\DefineSymbol{\proet}{pro\acute{e}t}
\DefineSymbol{\cond}{cond}
\DefineSymbol{\conti}{conti}
\DefineSymbol{\Cond}{Cond}
\DefineSymbol{\dCond}{dCond}
\DefineSymbol{\disc}{disc}
\DefineSymbol{\Tot}{Tot}
\DefineSymbol{\adic}{adic}
\DefineSymbol{\nc}{nc}
\DefineSymbol{\nuc}{nuc}
\DefineSymbol{\cpt}{cpt}
\DefineSymbol{\std}{std}
\DefineSymbol{\light}{light}
\DeclareMathOperator{\Hom}{Hom}
\DeclareMathOperator{\Map}{Map}

\DeclareMathOperator{\cofib}{cofib}
\DeclareMathOperator{\fib}{fib}

\DeclareMathOperator{\Spa}{Spa}

\DeclareMathOperator{\End}{End}
\DeclareMathOperator{\intHom}{\underline{Hom}}

\newcommand{\lra}{\longrightarrow}

\newcommand{\heart}{\heartsuit}

\newcommand{\hotimes}{\mathbin{\hat{\otimes}}}
\renewcommand{\Pr}{\mathcal{P}r}

\DefineCategory{\Set}{Set}
\DefineCategory{\Ab}{Ab}
\DefineCategory{\Ring}{Ring}
\DefineCategory{\Mod}{Mod}
\DefineCategory{\Alg}{Alg}
\DefineCategory{\Ch}{Ch}
\DefineCategory{\Mon}{Mon}
\DefineCategory{\CMon}{CMon}
\DefineCategory{\PCoh}{PCoh}
\DefineCategory{\Perf}{Perf}
\DefineCategory{\FP}{FP}
\DefineCategory{\Aff}{Aff}
\DefineCategory{\cAff}{cAff}
\DefineCategory{\AnRing}{AnRing}
\DefineCategory{\Ani}{Ani}
\DefineCategory{\CAlg}{CAlg}
\DefineCategory{\Prof}{Prof}

\newcommand{\Lbb}{\mathbb{L}}

\newcommand{\Nbb}{\mathbb{N}}

\newcommand{\Qbb}{\mathbb{Q}}

\newcommand{\Zbb}{\mathbb{Z}}

\newcommand{\Acal}{\mathcal{A}}
\newcommand{\Bcal}{\mathcal{B}}
\newcommand{\Ccal}{\mathcal{C}}
\newcommand{\Dcal}{\mathcal{D}}
\newcommand{\Ecal}{\mathcal{E}}

\newcommand{\Kcal}{\mathcal{K}}

\newcommand{\Ocal}{\mathcal{O}}

\newcommand{\Scal}{\mathcal{S}}

\renewcommand{\hat}{\widehat}

\everymath{\displaystyle}

\begin{document}

\title[]{fppf-descent for condensed animated rings}
\author{Yutaro Mikami}
\date{\today}
\address{Graduate School of Mathematical Sciences, University of Tokyo, 3-8-1 Komaba, Meguro-ku, Tokyo 153-8914, Japan}
\curraddr{RIKEN, Center for Advanced Intelligence Project AIP, Mathematical Science Team, 1-4-1 Nihonbashi, Chuo-ku, Tokyo, 103-0027, Japan}
\email{yutaro.mikami@riken.jp}
\maketitle
\begin{abstract}
In this paper, we define ``animated affinoid algebras'' and prove some basic properties of them.
Then we generalize the result of \cite{Mann22, Mikami22} (fppf-descent for discrete rings or affinoid algebras) to discrete animated rings or animated affinoid algebras.
\end{abstract}

\setcounter{tocdepth}{1}

\tableofcontents


\section{Introduction}
There are two goals of this paper.
\subsection{Animated affinoid $K$-algebras}
Let $K$ be a complete non-archimedean field.
The first goal is to give an ``animation'' of affinoid $K$-algebras in some sense by using condensed mathematics, which was introduced by Clausen-Scholze in \cite{CM}.
It will be necessary to define and study ``derived rigid geometry''.
Derived rigid geometry has been studied without using condensed mathematics by Porta-Yu and Ben-Bassat-Kremnizer \cite{PY18, PY20, PY21, BBK17}.
One of the advantages of using condensed mathematics is that it makes it possible to use a general theory of analytic spaces in the sense of Clausen-Scholze (\cite{CM, AG, CC}).
For example, we can naturally define quasi-coherent complexes over animated affinoid $K$-algebras, and we can construct a 6-functor formalism for them.

We can naively define an \textit{animated affinoid $K$-algebra} to be a solid condensed animated $K$-algebra $A$ such that $\pi_0A$ is an affinoid $K$-algebra in the usual sense (more precisely, there is an affinoid $K$-algebra $A^{\prime}$ in the usual sense such that the associated condensed ring $\underline{A^{\prime}}$ is isomorphic to $\pi_0A$). 
However this definition does not work well.
The problem is that for such a solid condensed animated $K$-algebra $A$, $(K,\Zbb)_{\square} \to (A,\Zbb)_{\square}$ is not necessarily steady (\cite[Definition 2.3.16]{Mann22}), where $(K,\Zbb)_{\square}$ and $(A,\Zbb)_{\square}$ are analytic animated rings whose analytic animated ring structures are induced from $\Zbb_{\square}$.
Because of this problem, tensor products of such solid condensed animated $K$-algebras and scalar extension functors are ill-behaved.
If $A$ is nuclear as an object of $\Dcal((K,\Zbb)_{\square})$, then $(K,\Zbb)_{\square} \to (A,\Zbb)_{\square}$ is steady by \cite[Proposition 13.14]{AG}.
Therefore, we find that the following definition of animated affinoid $K$-algebras is more natural.
\begin{definition}
An \textit{animated affinoid $K$-algebra} is a condensed animated $K$-algebra $A$ satisfying the following conditions:
\begin{enumerate}
\item
A static solid ring $\pi_0A$ is an affinoid $K$-algebra in the usual sense; that is, $\pi_0A$ is a Banach $K$-algebra topologically of finite type..
\item
The object $A$ of $\Dcal(\underline{K})$ is $(K,\Zbb)_{\square}$-complete and nuclear as an object of $\Dcal((K,\Zbb)_{\square})$. 
\end{enumerate}
\end{definition}

We will prove that animated affinoid $K$-algebras as in the above definition satisfy some basic properties as in the case of affinoid $K$-algebras in the usual sense.

\subsection{Fppf-descent}
The second goal is to give a generalization of the results of \cite{Mann22, Mikami22} (fppf-descent for discrete rings or affinoid $K$-algebras) to discrete animated rings or animated affinoid $K$-algebras.
Let $A\to B$ be a faithfully flat map of discrete animated rings, and let $B^{n/A}$ denote the $n$-fold derived tensor product of $B$ over $A$.
Lurie and Mathew proved the following result in \cite{SAG, Mat16}.
\begin{theorem}
We have an equivalence of $\infty$-categories 
$$\Dcal(A) \overset{\sim}{\lra} \varprojlim_{[n] \in \Delta} \Dcal((B^{n+1/A})), $$
where $\Dcal(A)$ is a derived $\infty$-category of $A$-modules.
\end{theorem}

In the above theorem, only discrete modules are considered.
We can also consider a generalization to topological modules (more precisely, solid condensed modules). 
The notion of solid condensed modules is better than that of discrete modules in some sense. 
For example there is a 6-functor formalism in the world of solid condensed modules.
Therefore, this generalization will be worthwhile.
However, it is difficult to prove descent for solid condensed modules.
One of the reasons is that the scalar extension functor $$-\otimes_{A_{\square}}^{\Lbb} B_{\square} \colon \Dcal(A_{\square})\to \Dcal(B_{\square})$$ is not given by the derived tensor product $-\otimes_{A_{\square}}^{\Lbb} B \colon  \Dcal(A_{\square})\to \Dcal(A_{\square})$.
Mann and the author independently proved descent for solid condensed modules in different ways when $A\to B$ be a finitely presented faithfully flat map of discrete (static) rings (\cite{Mann22, Mikami22}).

\begin{theorem}
Let $A\to B$ be a finitely presented faithfully flat map of discrete static rings.
Then we have an equivalence of $\infty$-categories 
$$\Dcal(A_{\square}) \overset{\sim}{\lra} \varprojlim_{[n] \in \Delta} \Dcal((B^{n+1/A})_{\square}), $$
where $\Dcal(A_{\square})$ is a derived $\infty$-category of $A_{\square}$-modules.
\end{theorem}

A key result of the proof of the author is that there exists a compact object $N_{B/A} \in \Dcal(A_{\square})$ such that there is an equivalence of functors from $\Dcal(A_{\square})$ to $\Dcal(A_{\square})$
$$-\otimes_{A_{\square}}^{\Lbb} B_{\square} \simeq R\intHom_A(N_{B/A},-),$$
where the functor $-\otimes_{A_{\square}}^{\Lbb} B_{\square} \colon \Dcal(A_{\square}) \to \Dcal(A_{\square})$ is the composition of $$-\otimes_{A_{\square}}^{\Lbb} B_{\square} \colon \Dcal(A_{\square}) \to \Dcal(B_{\square})$$ with the forgetful functor $\Dcal(B_{\square})\to\Dcal(A_{\square})$.
In this paper, we will prove the following theorem from the above result.

\begin{theorem}\label{thm:introN}
Let $A\to B$ be a faithfully flat map of discrete animated rings such that $\pi_0f \colon \pi_0A \to \pi_0B$ is finitely presented.
Then there exists a compact object $N_{B/A} \in \Dcal(A_{\square})$ such that there is an equivalence of functors from $\Dcal(A_{\square})$ to $\Dcal(A_{\square})$
$$-\otimes_{A_{\square}}^{\Lbb} B_{\square} \simeq R\intHom_{A}(N_{B/A},-).$$
\end{theorem}

In the proof of Theorem \ref{thm:introN}, it is a key claim that there exists a compact object $N_{B/A}\in \Dcal(A_{\square})$ which represents a functor 
$$\Dcal(A_{\square}) \to \Ani ;\; M \mapsto \Map_{A}(A, M\otimes_{A_{\square}}^{\Lbb} B_{\square}),$$
where $\Map_{A}(A, M\otimes_{A_{\square}}^{\Lbb} B_{\square})$ is the mapping anima from $A$ to $M\otimes_{A_{\square}}^{\Lbb} B_{\square}$ in $\Dcal(A_{\square})$ (Lemma \ref{lem:repani}).
To prove this claim, we will prove that the functor $$-\otimes_{A_{\square}}^{\Lbb} B_{\square} \colon \Dcal(A_{\square}) \to \Dcal(A_{\square})$$ preserves small limits by comparing with the limit-preserving functor 
$$-\otimes_{\pi_0A_{\square}}^{\Lbb} \pi_0B_{\square} \colon \Dcal(\pi_0A_{\square}) \to \Dcal(\pi_0A_{\square}),$$
and then we will use the adjoint functor theorem (\cite[Corollary 5.5.2.9]{HTT}).

From uniform boundedness of $\{N_{\pi_0B/\pi_0A}^{\otimes n}\}_{n\geq 1}$ (\cite[Proposition 2.20]{Mikami22}), we can prove that $\{N_{B/A}^{\otimes n}\}_{n\geq 1}$ is uniformly bounded above (Theorem \ref{thm:unifbdd}).
From a formal argument, we find that it implies that $A_{\square} \to B_{\square}$ is descendable in the sense of \cite{Mann22}.
In other words, we get the following main theorem.

\begin{theorem}
Let $A\to B$ be a faithfully flat map of discrete animated rings such that $\pi_0f \colon \pi_0A \to \pi_0B$ is finitely presented.
Then we have an equivalence of $\infty$-categories 
$$\Dcal(A_{\square}) \overset{\sim}{\lra} \varprojlim_{[n] \in \Delta} \Dcal((B^{n+1/A})_{\square}).$$
\end{theorem}

By the similar argument, we can also prove faithfully flat descent for animated affinoid $K$-algebras.

\begin{theorem}\label{thm:intro1}
Let $f \colon A\to B$ be a faithfully flat morphism of animated affinoid $K$-algebras.
Let $B^{n/A}$ denote the $n$-fold derived tensor product of $B$ over $(A,\Zbb)_{\square}$.
Then we have an equivalence of $\infty$-categories 
$$\Dcal(A_{\square}) \overset{\sim}{\lra} \varprojlim_{[n] \in \Delta} \Dcal((B^{(n+1)/A})_{\square}). $$
\end{theorem}

\subsection*{Outline of the paper}
This paper is organized as follows.
In the first part of Section 2, we will prove basic properties of nuclear $(A,\Zbb)_{\square}$-modules for a complete adic ring $A$.
In the second part of Section 2, we will define animated affinoid $K$-algebras and animated affinoid pairs, and we will prove basic properties.
In Section 3, we will prove that $\{N_{B/A}^{\otimes n}\}_{n\geq 1}$ is uniformly bounded above for a faithfully flat map $A\to B$ of affinoid $K$-algebras, which will be necessary to prove Theorem \ref{thm:intro1}.
In the first part of Section 4, we will compare the descent result in \cite{Mikami22} and \cite{Mann22}, and we will prove formal results about descent for analytic animated rings.
In the second part of Section 4, we will prove the main theorem.

\subsection*{Convention}
\begin{itemize}
\item
All rings, including condensed ones, are assumed unital and commutative.
\item
For an $\infty$-category $\Ccal$, $0$-truncated objects of $\Ccal$ are called \textit{discrete objects} in \cite{HTT}.
However, this term conflicts with the term ``discrete" in the topological sense, so we use the term \textit{static object} to refer to an \textit{$0$-truncated object}.
\item In this paper, we work in the context of \textit{light} condensed mathematics.
For the foundations of light condensed mathematics, see \cite{RC26}.
\item 
In contrast to \cite{Mann22}, we use the term \textit{ring} to refer to an \textit{ordinary ring} (not a condensed animated ring). 
Sometimes we use the term \textit{discrete ring} (resp. \textit{discrete animated ring}) to refer to an ordinary ring (resp. animated ring) in order to emphasize that it is not a condensed one.
We also use the term \textit{static ring} (resp. \textit{static analytic ring})  to refer to an ordinary ring (resp. analytic ring) in order to emphasize that it is not an animated one.
\item 
We often identify a compactly generated topological ring $A$ with a condensed ring $\underline{A}$ associated to $A$.
It is justified by \cite[Proposition 1.7]{CM}.
If there is no room for confusion, we simply write $A$ for $\underline{A}$.
\item
We use the terms ``analytic animated ring'' and ``uncompleted analytic animated ring'' according to \cite{Mann22}.
\item
For an uncompleted analytic animated ring $\Acal$, we denote the underlying condensed animated ring of $\Acal$ by $\underline{\Acal}$.
\item
For an uncompleted analytic animated ring $\Acal$, an object $M\in \Dcal(\underline{\Acal})$ is said to be \textit{$\Acal$-complete} if it lies in $\Dcal(\Acal)$.
\item
We use the terms \textit{f-adic ring} and \textit{affinoid pair} rather than \textit{Huber ring} and \textit{Huber pair}.
\item
For an f-adic ring $A$, we denote the ring of power-bounded elements of $A$ by $A^{\circ}$.
\item For a complete non-archimedean field $K$, we use the term \textit{affinoid $K$-algebra} to refer to a \textit{Banach $K$-algebra topologically of finite type over $K$}.
\item
We let $\Ani$ denote the $\infty$-category of anima. It is equivalent to the $\infty$-category $\Scal$ of spaces (\cite[Definition 1.2.16.1]{HTT}).
\item
We let $\Pr^L$ denote the $\infty$-category of presentable $\infty$-categories with colimit preserving functors. 
\item
For an $\infty$-category $\Ccal$ and objects $X,Y\in \Ccal$, we denote the mapping anima from $X$ to $Y$ by $\Map_{\Ccal}(X,Y)$ or $\Map(X,Y)$ if $\Ccal$ is clear from the context.
\item
We denote the simplex category by $\Delta$, which is the full subcategory of the category of totally ordered sets consisting of the totally ordered sets $[n]=\{0,\ldots,n\}$ for all $n \geq 0$. 
We also denote the subcategory of $\Delta$ with the same objects but where the morphisms are given by injective maps by $\Delta_s$.
Moreover, for every $m \geq 0$ we denote the full subcategory of $\Delta_s$ consisting of $[n]$ for all $0 \leq n \leq m$ by $\Delta_{s,\leq m}$.
\end{itemize}
\subsection*{Acknowledgements}
The author is grateful to Yoichi Mieda for his support during the studies of the author.
In addition, the author is grateful to Grigory Andreychev and Lucas Mann for their comments on this paper.
This work was supported by JSPS KAKENHI Grant Number JP23KJ0693, and the RIKEN Special Postdoctoral Researcher Program.
This paper was written during a visit to the Hausdorff Research Institute for Mathematics, funded by the Deutsche Forschungsgemeinschaft (DFG, German Research Foundation) under Germany's Excellence Strategy – EXC-2047/1 – 390685813.


\section{Preliminaries on light condensed mathematics}
Although the material in this chapter is well-known, we present it here because we could not find any reference written in the setting of light condensed mathematics.

\begin{definition}
	Let $\Prof^{\light}$ denote the site of light profinite sets.
	For a presentable $\infty$-category $\Ccal$, let $\Cond(\Ccal)$ denote the $\infty$-category of hypercomplete $\Ccal$-valued sheaves, which we will refer to as the $\infty$-category of (light) condensed objects in $\Ccal$.
\end{definition}

\begin{lemma}\label{lem:discrete fully faithful}
	Let $\Ccal$ be a compactly generated $\infty$-category\footnote{The author does not know whether this lemma holds for general presentable $\infty$-categories.}.
	Then the functor 
	$$\Cond(\Ccal)\to \Ccal;\; F\mapsto F(\ast)$$ 
	admits a fully faithful left adjoint functor $\Ccal \to \Cond(\Ccal);\; X\mapsto X^{\cond}$.
\end{lemma}
\begin{proof}
	We need to show that the left adjoint functor $\Ccal \to \Cond(\Ccal)$ is fully faithful.
	We take a strong limit cardinal $\kappa_1$ such that $\kappa_1>\aleph_1$, and a regular cardinal $\kappa_2$ such that $\kappa_2>\kappa_1$.
	Then there are natural functors
	$$\Ccal=\Cond(\Ccal)_{\aleph_0}\to \Cond(\Ccal)=\Cond(\Ccal)_{\aleph_1}\to \Cond(\Ccal)_{\kappa_1}\to \Cond(\Ccal)_{\kappa_2},$$
	where $\Cond(\Ccal)_{\kappa}$ is defined in \cite[Definitions 2.1.1 and 2.1.3]{Mann22}.
	Then the functor $\Cond(\Ccal)_{\aleph_0}\to \Cond(\Ccal)_{\kappa_1}$ is fully faithful by \cite[Proposition 2.5]{aoki24}.
	Moreover, by \cite[Proposition 2.1.8]{Mann22}, $\Cond(\Ccal)_{\aleph_1}\to  \Cond(\Ccal)_{\kappa_2}$ is also fully faithful.
	Therefore, $\Ccal \to \Cond(\Ccal)$ is also fully faithful.
\end{proof}

\begin{definition}
	Let $\Ccal$ be a compactly generated $\infty$-category.
	We say that an object $X\in \Cond(\Ccal)$ is \textit{discrete} if $X$ lies in the essential image of the functor $\Ccal \to \Cond(\Ccal);\; X\mapsto X^{\cond}$ in Lemma \ref{lem:discrete fully faithful}.
	We often identify objects in $\Ccal$ with discrete objects in $\Cond(\Ccal)$.
	We call the functor $\Cond(\Ccal)\to \Cond(\Ccal);\; F\mapsto F(\ast)^{\cond}$ the \textit{discretization functor}.
\end{definition}

\begin{example}
	\begin{enumerate}
		\item The $\infty$-category $\Ani\Ring$ of animated rings is compactly generated.
		Therefore, we can regard an animated ring as a discrete condensed animated ring via the fully faithful functor 
		$$\Ani\Ring \to \Cond(\Ani\Ring);\; A\mapsto A^{\cond}.$$
		\item The derived $\infty$-category $\Dcal(\Ab)$ of $\Ab$ is compactly generated.
		Therefore,  we can regard an object of $\Dcal(\Ab)$ as a discrete object of $\Cond(\Dcal(\Ab))$ via the fully faithful functor 
		$$\Dcal(\Ab) \to \Cond(\Dcal(\Ab));\; M\mapsto M^{\cond}.$$
		Note that, by \cite[Theorem 2.1.2.2]{SAG}, there is a natural equivalence of $\infty$-categories $$\Dcal(\Cond(\Ab))\simeq \Cond(\Dcal(\Ab)).$$
		Under this equivalence, the functor $\Dcal(\Cond(\Ab))\to \Dcal(\Ab);\;M\mapsto M(\ast)$ (resp. $\Dcal(\Ab)\to \Dcal(\Cond(\Ab));\; N\mapsto N^{\cond}$) is the derived functor of the exact functor $\Cond(\Ab)\to \Ab;\;M\mapsto M(\ast)$ (resp. $\Ab\to \Cond(\Ab);\; N\mapsto N^{\cond}$).
	\end{enumerate}
\end{example}

\section{Animated affinoid $K$-algebras}
We begin with the construction of an uncompleted analytic animated ring from a pair of a condensed animated ring and a ring.

\begin{definition}\label{def:ana}
Let $A$ be a condensed animated ring, and let $A_0 \to \pi_0(A)(\ast)$ be a morphism of (usual) rings.
Then we get an analytic ring $A_{0,\square}=(A_0, A_0)_{\square}$ (\cite[Definition 2.9.3]{Mann22}).
We define an uncompleted analytic animated ring $(A,A_0)_{\square}$ as follows:
If $A$ is static then there is a natural morphism $A_0^{\cond} \to A$ and we define an uncompleted analytic animated ring $(A,A_0)_{\square}$ as the condensed animated ring $A$ with the induced uncompleted analytic animated ring structure from $A_{0,\square}$ (\cite[Proposition 12.8]{AG} or an uncompleted version of \cite[Definition 2.3.13]{Mann22}).
For general $A$, we define $(A,A_0)_{\square}$ as the condensed animated ring $A$ with an uncompleted analytic animated ring structure induced from $(\pi_0(A),A_0)_{\square}$ by using \cite[Proposition 12.21]{AG}.
\end{definition}

\begin{remark}\label{rem:sta}
By the construction, an object $M\in \Dcal(A)$ is $(A,A_0)_{\square}$-complete if and only if the $\pi_0A$-module $H^i(M)$ is $A_{0,\square}$-complete (equivalently,  $H^i(M)$ is $(\pi_0A, A_0)_{\square}$-complete) for every $i\in \Zbb$.
In particular, if for every non-negative integer $n$, $\pi_n(A)$ is $A_{0,\square}$-complete (for example, $A$ can be written as a limit of discrete objects), then $(A,A_0)_{\square}$ is an analytic animated ring.
\end{remark}

Next, let us recall the definition of (basic) nuclear objects.

\begin{definition}[{\cite[Definition 8.1, Definition 8.5]{CC}}]
	Let $\Acal$ be an analytic animated $\Zbb_{\square}$-algebra\footnote{We assume that $\Acal$ is a $\Zbb_{\square}$-algebra to ensure that for any light profinite set $S$, $\Acal[S]\in \Dcal(\Acal)$ is compact.}.
	\begin{enumerate}
		\item An object $M\in \Dcal(\Acal)$ is \textit{nuclear} if for any light profinite set $S$, the natural morphism 
		$$(R\intHom_{\Acal}(\Acal[S],\underline{\Acal}) \otimes_{\Acal}^{\Lbb} M)(\ast) \to R\Hom_{\Acal}(\Acal[S],M)$$
		is an equivalence.
		\item A morphism $f\colon M\to N$ in $\Dcal(\Acal)$ is \textit{trace-class} if $f$ lies in the image of the natural morphism
		$$H^0((R\intHom_{\Acal}(M,\underline{\Acal})\otimes_{\Acal}^{\Lbb}N)(\ast))\to \Hom(M,N).$$
		\item An object $M\in \Dcal(\Acal)$ is \textit{basic nuclear} if it is equivalent to the colimit of a sequence
		$$M_0\to M_1\to M_2\to \cdots$$
		of trace-class morphisms between objects of $\Dcal(\Acal)$.
	\end{enumerate}
\end{definition}

\begin{remark}
	Since the $\infty$-subcategory of compact objects of $\Dcal(\Acal)$ is generated by the objects $\Acal[S]$ under finite colimits and retracts, an object $M\in \Dcal(\Acal)$ is nuclear if and only if for any compact object $C\in\Dcal(\Acal)$, the natural morphism 
		$$\alpha \colon (R\intHom_{\Acal}(C,\underline{\Acal}) \otimes_{\Acal}^{\Lbb} M)(\ast) \to R\Hom_{\Acal}(C,M)$$
		is an equivalence.
\end{remark}

\begin{remark}
	By the proof of \cite[Theorem 8.6]{CC}, a basic nuclear object $M\in\Dcal(\Acal)$ is equivalent to the colimit of a sequence
	$$M_0\to M_1\to M_2\to \cdots$$
	of trace-class morphisms between compact objects of $\Dcal(\Acal)$.
\end{remark}

\begin{lemma}
Let $\Acal$ be an analytic animated $\Zbb_{\square}$-algebra, and $M\in \Dcal(\Acal)$ be an object of $\Dcal(\Acal)$.
The following are equivalent:
\begin{enumerate}
	\item The object $M$ is nuclear.
	\item For any light profinite set $S$, the natural morphism 
		$$R\intHom_{\Acal}(\Acal[S],\underline{\Acal}) \otimes_{\Acal}^{\Lbb} M\to R\intHom_{\Acal}(\Acal[S],M)$$
		is an equivalence.
	\item For any compact object $C\in \Dcal(\Acal)$ and an object $N\in \Dcal(\Acal)$, the natural morphism 
	$$R\intHom_{\Acal}(C,N) \otimes_{\Acal}^{\Lbb} M\to R\intHom_{\Acal}(C,N\otimes_{\Acal}^{\Lbb}M)$$
	is an equivalence.
\end{enumerate}
\end{lemma}
\begin{proof}
It is clear that (2) implies (1), and (3) implies (2). That (1) implies (3) can be proved similarly to \cite[Proposition 5.35]{And21}.
\end{proof}

Let $A$ be a $t$-adically complete condensed animated ring such that $A/^{\Lbb}t$ is discrete ($t\in \pi_0(A)$), where we let $A/^{\Lbb}t$ denote $A\otimes_{\Zbb[T]}^{\Lbb} \Zbb$. 
We will give a characterization of nuclear objects over $A$ based on ideas used in \cite{Mann22b}.
Some of the results below are also proved in the thesis of Andreychev (\cite{And23}).

\begin{definition}
An object $M \in \Dcal(A)$ is said to be \textit{Banach} over $A$ if $M$ is $t$-adically complete and $M/^{\Lbb}t$ is discrete.
\end{definition}

\begin{lemma}\label{lem:trunc}
The full $\infty$-subcategory of $\Dcal(A)$ spanned by Banach objects of $\Dcal(A)$ is closed under truncations on $\Dcal(A)$.
\end{lemma}
\begin{proof}
Let $M\in \Dcal(A)$ be a Banach object.
We will prove that $\tau^{0\geq}M$ and $\tau^{1\leq}M$ are Banach objects.
By \cite[Lemma 2.12.4]{Mann22}, $\tau^{0\geq}M$ and $\tau^{1\leq}M$ are $t$-adically complete, so it is enough to show that $(\tau^{0\geq}M)/^{\Lbb}t$ and $(\tau^{1\leq}M)/^{\Lbb}t$ are discrete.
We have a fiber sequence in $\Cond(\Dcal(\Ab))$
$$(\tau^{0\geq}M)/^{\Lbb}t \to M/^{\Lbb}t \to (\tau^{1\leq}M)/^{\Lbb}t.$$
Since $\tau^{-1\geq}\left((\tau^{0\geq}M)/^{\Lbb}t\right) \simeq \tau^{-1\geq}(M/^{\Lbb}t)$ and $\tau^{1\leq}\left((\tau^{1\leq}M)/^{\Lbb}t\right) \simeq \tau^{1\leq}(M/^{\Lbb}t)$ are discrete, it is enough to show that $M_0\coloneqq H^0\left((\tau^{0\geq}M)/^{\Lbb}t\right)$ and
$M_1\coloneqq H^0\left((\tau^{1\leq}M)/^{\Lbb}t\right)$ are discrete.
Let $M_{0,\disc}$ and $M_{1,\disc}$ denote the discretizations of $M_0$ and $M_1$.
Then we have the following commutative diagram of static condensed $\Zbb$-modules:
$$\xymatrix{
0\ar[r] & M_{0,\disc} \ar[r]\ar[d] & H^0(M/^{\Lbb}t) \ar[r] \ar[d]^{\cong} &M_{1,\disc} \ar[r]\ar[d]& 0 \\
0 \ar[r]& M_0 \ar[r]& H^0(M/^{\Lbb}t) \ar[r]&M_1 \ar[r]& 0, \\
}
$$
where the horizontal sequences are exact.
Therefore, it is enough to show that for any static condensed abelian group $N\in (\Cond(\Dcal(\Ab)))^{\heart}$, the natural morphism $N_{\disc} \to N$ is injective.
Let $S$ be any light profinite set. We write $S$ as a countable cofiltered limit $ S=\varprojlim_{n\in \Nbb} S_{n}$ where $S_{n}$ is a finite set and the natural maps $S\to S_{n}$ are surjective. 
Then the map $N_{\disc}(S) =\varinjlim_{n} N(S_n) \hookrightarrow N(S)$ is injective.
Therefore, $N_{\disc} \to N$ is also injective.
\end{proof}

\begin{proposition}\label{bannuc}
Let $M\in \Dcal(A)$ be an object which can be written as a filtered colimit of Banach objects of $\Dcal(A)$.
Then for every subring $A_0$ of the ring $\pi_0(A)(\ast)$, $M$ is $(A,A_0)_{\square}$-complete and nuclear as an object of $\Dcal((A,A_0)_{\square})$.
\end{proposition}
\begin{proof}
It is enough to show that a Banach object $M\in \Dcal(A)$ is $(A,A_0)_{\square}$-complete and nuclear as an object of $\Dcal((A,A_0)_{\square})$.
Since $M$ can be written as a limit of discrete objects in $\Dcal(A)$ and every discrete object in $\Dcal(A)$ is $(A,A_0)_{\square}$-complete, $M$ is also $(A,A_0)_{\square}$-complete.
We will prove that $M$ is a nuclear object of $\Dcal((A,A_0)_{\square})$.
Since $ M\simeq \varinjlim_{n}\tau^{n\geq}M$ and $\tau^{n\geq}M$ is Banach by Lemma \ref{lem:trunc}, we may assume that $M$ is bounded above.
Let $S$ be a light profinite set. 
We will show that 
$$\alpha \colon R\intHom_A((A,A_0)_{\square}[S],A) \otimes_{(A,A_0)_{\square}}^{\Lbb} M \to R\intHom_A((A,A_0)_{\square}[S],M)$$
is an equivalence.
Since both sides are $t$-adically complete by \cite[Proposition 2.12.10]{Mann22}, so it suffices to show that 
\begin{align*}
\alpha/^{\Lbb}t \colon &R\intHom_{A/^{\Lbb}t}((A/^{\Lbb}t,A_0)_{\square}[S],A/^{\Lbb}t) \otimes_{(A/^{\Lbb}t,A_0)_{\square}}^{\Lbb} M/^{\Lbb}t \\
&\to R\intHom_{A/^{\Lbb}t}((A/^{\Lbb}t,A_0)_{\square}[S],M/^{\Lbb}t)
\end{align*}
 is an equivalence.
By \cite[Proposition 2.9.7]{Mann22}, the discrete object $M/^{\Lbb}t \in \Dcal(A/^{\Lbb}t)$ is nuclear as an object of $\Dcal((A/^{\Lbb}t, A_0)_{\square})$, so we find that $\alpha/^{\Lbb}t$ is an equivalence.
\end{proof}

\begin{proposition}\label{nucban}
Let $M \in \Dcal((A,A_0)_{\square})$ be a nuclear object of $\Dcal((A,A_0)_{\square})$.
Then $M$ can be written as a filtered colimit of Banach objects of $\Dcal(A)$.
\end{proposition}
\begin{proof}
Since every nuclear object of $\Dcal((A,A_0)_{\square})$ can be written as a filtered colimit of basic nuclear objects of $\Dcal((A,A_0)_{\square})$, 
we may assume that $M$ is a basic nuclear objects $\varinjlim_{n} P_n$ of $\Dcal((A,A_0)_{\square})$, where each $P_n$ is compact as an object of $\Dcal((A,A_0)_{\square})$ and each transition morphism $P_n\to P_{n+1}$ is trace-class.
From the definition of trace-class maps and $P_n^{\vee} \simeq \hat{P_n}^{\vee}$, where $\hat{P_n}$ is the $t$-adic completion of $P_n$ and $\hat{P_n}^{\vee}$ is the dual $R\intHom_A(\hat{P_n},A)$ of $\hat{P_n}$, the transition morphisms $P_n \to P_{n+1}$ factor through $\hat{P_n}$, so we get that $M \simeq \varinjlim_{n} \hat{P_n}$.
By construction, $\hat{P_n} \to \hat{P_{n+1}}$ is a trace-class map. It is enough to show that $\hat{P_n} \to \hat{P_{n+1}}$ factors through a Banach object of $\Dcal(A)$.
We denote $\hat{P_n}\to \hat{P_{n+1}}$ by $f\colon P\to Q$. 
We define a Banach object $Q_b=R\varprojlim_{n} (Q/^{\Lbb}t^n)_{\disc}\in \Dcal(A)$, where $(Q/^{\Lbb}t^n)_{\disc}$ is the discretization of $Q/^{\Lbb}t^n$.
We will prove that $f$ factors through $Q_b$.
Since $f$ is trace-class, $f$ corresponds to a morphism $A\to P^{\vee} \otimes_{(A,A_0)_{\square}}^{\Lbb} Q$.
It is enough to show that 
\begin{align}
R\Hom_A (A,P^{\vee}\otimes_{(A,A_0)_{\square}}^{\Lbb} Q_b) \to R\Hom_A (A,P^{\vee}\otimes_{(A,A_0)_{\square}}^{\Lbb} Q)
\label{eq1}
\end{align}
is an equivalence.
We may assume that $P=\hat{(A,A_0)_{\square}[S]}$ for some light profinite set $S$.
Since both sides of (\ref{eq1}) are $t$-adically complete by \cite[Proposition 2.12.10]{Mann22}, it suffices to show that 
\begin{align}
R\Hom_{\Zbb} (\Zbb, (P/^{\Lbb}t)^{\vee}\otimes_{(A/^{\Lbb}t,A_0)_{\square}}^{\Lbb} (Q/^{\Lbb}t)_{\disc}) \to R\Hom_{\Zbb} (\Zbb,(P/^{\Lbb}t)^{\vee}\otimes_{(A/^{\Lbb}t,A_0)_{\square}}^{\Lbb} Q/^{\Lbb}t)
\label{eq2}
\end{align}
 is an equivalence.
We take a countable set $I$ such that there exists an isomorphism $\Zbb_{\square}[S]\cong \prod_I \Zbb$.
Then $(P/^{\Lbb}t)^{\vee}$ is equivalent to $R\intHom_{\Zbb}(\prod_I \Zbb , A/^{\Lbb}t)\simeq \bigoplus_I A/^{\Lbb}t$, so the morphism (\ref{eq2}) is equivalent to 
\begin{align*}
R\Hom_{\Zbb} (\Zbb, \bigoplus_I (Q/^{\Lbb}t)_{\disc}) \to R\Hom_{\Zbb} (\Zbb, \bigoplus_I Q/^{\Lbb}t),
\end{align*}
which is an equivalence.
\end{proof}

\begin{corollary}\label{cor:nuc}
Let $A_0, A_0^{\prime}$ be subrings of the ring $\pi_0(A)(\ast)$.
Let $M$ be an object of $\Dcal(A)$.
Then $M$ is a nuclear object of $\Dcal((A,A_0)_{\square})$ if and only if $M$ is a nuclear object of $\Dcal((A,A_0^{\prime})_{\square})$.
\end{corollary}

We note the following lemma.

\begin{lemma}\label{lem:nuccoef}
Let $\Acal \to \Bcal$ be a morphism of analytic animated $\Zbb_{\square}$-algebras.
We assume that the analytic animated ring structure of $\Bcal$ is induced from $\Acal$ and that $\underline{\Bcal}$ is nuclear as an object of $\Dcal(\Acal)$.
Then, $M \in \Dcal(\Bcal)$ is nuclear as an object of $\Dcal(\Acal)$ if and only if $M$ is nuclear as an object of $\Dcal(\Bcal)$.
\end{lemma}
\begin{proof}
Let $S$ be a light profinite set.
First, we assume that $M$ is nuclear as an object of $\Dcal(\Bcal)$.
Then we have
\begin{align*}
&\phantom{{}\simeq {}}R\intHom_{\Acal} (\Acal[S], M) \\
&\simeq R\intHom_{\Bcal} (\Bcal[S], M) \\
&\simeq R\intHom_{\Bcal} (\Bcal[S], \underline{\Bcal})\otimes_{\Bcal}^{\Lbb} M \\
&\simeq R\intHom_{\Acal} (\Acal[S], \underline{\Bcal})\otimes_{\Bcal}^{\Lbb} M \\
&\simeq \left(R\intHom_{\Acal} (\Acal[S], \underline{\Acal})\otimes_{\Acal}^{\Lbb} \underline{\Bcal} \right)\otimes_{\Bcal}^{\Lbb} M \\
&\simeq R\intHom_{\Acal} (\Acal[S], \underline{\Acal})\otimes_{\Acal}^{\Lbb} M,
\end{align*} 
where the second equivalence follows from nuclearity of $M$ as an object of $\Dcal(\Bcal)$ and the fourth equivalence follows from nuclearity of $\underline{\Bcal}$ as an object of $\Dcal(\Acal)$, 
so $M$ is nuclear as an object of $\Dcal(\Acal)$.

Second, we assume that $M$ is nuclear as an object of $\Dcal(\Acal)$.
Then we have
\begin{align*}
&\phantom{{}\simeq {}}R\intHom_{\Bcal} (\Bcal[S], \underline{\Bcal})\otimes_{\Bcal}^{\Lbb} M \\
&\simeq R\intHom_{\Acal} (\Acal[S], \underline{\Bcal})\otimes_{\Bcal}^{\Lbb} M \\
&\simeq \left(R\intHom_{\Acal} (\Acal[S], \underline{\Acal})\otimes_{\Acal}^{\Lbb} \underline{\Bcal}\right)\otimes_{\Bcal}^{\Lbb} M \\
&\simeq R\intHom_{\Acal} (\Acal[S], \underline{\Acal})\otimes_{\Acal}^{\Lbb} M\\
&\simeq R\intHom_{\Acal} (\Acal[S], M) \\
&\simeq R\intHom_{\Bcal} (\Bcal[S], M),
\end{align*} 
where the second equivalence follows from nuclearity of $\underline{\Bcal}$ as an object of $\Dcal(\Acal)$ and the fourth equivalence follows from nuclearity of $M$ as an object of $\Dcal(\Acal)$, 
so $M$ is nuclear as an object of $\Dcal(\Bcal)$.
\end{proof}

\begin{proposition}\label{prop:nucsta}
Let $A_0$ be a subring of the ring $\pi_0(A)(\ast)$.
Then an object $M$ of $\Dcal((A,A_0)_{\square})$ is nuclear if and only if $H^i(M)$ is nuclear for every $i\in \Zbb$.
\end{proposition}
\begin{proof}
The ``only if'' direction follows from Proposition \ref{nucban} and Lemma \ref{lem:trunc}.
We will prove the ``if'' direction.
We may assume $(A,A_0)_{\square}=\Zbb[[t]]_{\square}$ by Corollary \ref{cor:nuc} and Lemma \ref{lem:nuccoef}.
It is enough to show that for every light profinite set $S$, functors $\Zbb[[t]]_{\square}[S]^{\vee} \otimes_{\Zbb[[t]]_{\square}}^{\Lbb} -$ and $R\intHom_{\Zbb[[t]]}(\Zbb[[t]]_{\square}[S],-)$ commute with Postnikov limits, where $\Zbb[[t]]_{\square}[S]^{\vee}=R\intHom_{\Zbb[[t]]}(\Zbb[[t]]_{\square}[S],\Zbb[[t]])$.
It is clear that the latter functor commute with Postnikov limits.
We will prove that the functor $\Zbb[[t]]_{\square}[S]^{\vee} \otimes_{\Zbb[[t]]_{\square}}^{\Lbb} -$ has finite Tor dimension $0$ (that is, $\Zbb[[t]][S]^{\vee}$ is a flat $\Zbb[[t]]_{\square}$-module).
We take a countable set $I$ such that there is an isomorphism $\Zbb_{\square}[S] \simeq \prod_{I} \Zbb$.
By \cite[Corollary 3.4.3]{RC26}, $\prod_{I} \Zbb$ is a flat $\Zbb_{\square}$-module.
Therefore, $\prod_I \Zbb[[t]]=(\prod_{I} \Zbb)\otimes_{\Zbb_{\square}}\Zbb[[t]]$ is also a flat $\Zbb[[t]]_{\square}$-module.
We have $$\Zbb[[t]]_{\square}[S]^{\vee} \simeq R\intHom_{\Zbb[[t]]}(\prod_I \Zbb[[t]],\Zbb[[t]]) \simeq \hat{\bigoplus_I} \Zbb[[t]],$$
where $\hat{\bigoplus_I} \Zbb[[t]]$ is the $t$-adic completion of $\bigoplus_I \Zbb[[t]]$.
It can be written as a filtered colimit of $\prod_I \Zbb[[t]]$, so we get the claim.
\end{proof}

Let us introduce the notion of animated affinoid algebras.
Let $K$ be a complete non-archimedean field, and let $\Ocal_{K}$ denote the ring of integers of $K$, and $\pi \in K$ denote a pseudo-uniformizer of $K$.
\begin{definition}\label{def:aff}
An \textit{animated affinoid $K$-algebra} is a condensed animated $K$-algebra $A$ satisfying the following conditions:
\begin{enumerate}
\item
The condensed $K$-algebra $\pi_0A$ is an affinoid $K$-algebra in the usual sense; that is, $\pi_0A$ is a Banach $K$-algebra topologically of finite type.
\item
The object $A$ of $\Dcal(\underline{K})$ is $(K,\Zbb)_{\square}$-complete and nuclear as an object of $\Dcal((K,\Zbb)_{\square})$. 
\end{enumerate}
A morphism of animated affinoid $K$-algebras is a morphism of condensed animated $K$-algebras.
\end{definition}

\begin{remark}
Let $A$ be an affinoid $K$-algebra in the usual sense.
We take a ring of definition $A_0\subset A$ which is $\Ocal_K$-algebra.
Then $A_0$ is a Banach $\Ocal_K$-module, so $A_0$ is nuclear as an object of $\Dcal((\Ocal_K,\Zbb)_{\square})$.
Since $A$ can be written as a colimit of $A_0$, $A$ is also nuclear as an object of $\Dcal((\Ocal_K,\Zbb)_{\square})$.
Therefore, by Lemma \ref{lem:nuccoef}, we find that $A$ is also nuclear as an object of $\Dcal((K,\Zbb)_{\square})$, where we note that $K$ is nuclear as an object of $\Dcal((\Ocal_K,\Zbb)_{\square})$.
In particular, an affinoid $K$-algebra is an animated affinoid $K$-algebra.
\end{remark}

\begin{remark}
Let $A$ be an animated affinoid $K$-algebra.
Then by Lemma \ref{lem:nuccoef}, $A$ is nuclear as an object of $\Dcal((\Ocal_K,\Zbb)_{\square})$.
By Corollary \ref{cor:nuc}, $A$ is also nuclear as an object of $\Dcal((\Ocal_K,\Ocal_K)_{\square})$, so $A$ is also nuclear as an object of $\Dcal((K,\Ocal_K)_{\square})$ by Lemma \ref{lem:nuccoef}.
\end{remark}

\begin{proposition}\label{prop:tensor prod of animated affinoid}
Let $B\leftarrow A \to C$ be a diagram of animated affinoid $K$-algebras.
Then $B\otimes_{(A,\Zbb)_{\square}}^{\Lbb} C$ is also an animated affinoid $K$-algebras
\end{proposition}
\begin{proof}
By using the Bar resolution of $A$ over $K$, we get that $$B\otimes_{(A,\Zbb)_{\square}}^{\Lbb} C \simeq \varinjlim_{n \in \Delta} (B \otimes_{(K,\Zbb)_{\square}}^{\Lbb} A^{\otimes n-1} \otimes_{(K,\Zbb)_{\square}}^{\Lbb} C),$$
so $B\otimes_{(A,\Zbb)_{\square}}^{\Lbb} C$ is nuclear as an object of $\Dcal((K,\Zbb)_{\square})$, where we note that a derived tensor product of nuclear objects is nuclear (\cite[Theorem 8.6]{CC}).
Next we will prove that $\pi_0(B\otimes_{(A,\Zbb)_{\square}}^{\Lbb} C)$ is an affinoid $K$-algebra.
Since we have $\pi_0(B\otimes_{(A,\Zbb)_{\square}}^{\Lbb} C)\simeq \pi_0B\otimes_{(\pi_0A,\Zbb)_{\square}} \pi_0C$, we may assume that $A$, $B$, and $C$ are affinoid $K$-algebras.
We put $B=A\langle X_1,\ldots,X_m\rangle/(f_1,\ldots,f_n)$.
Then we have an isomorphism
$$(A\langle X_1,\ldots,X_m\rangle/(f_1,\ldots,f_n)) \otimes_{(A,\Zbb)_{\square}}C \simeq (A\langle X_1,\ldots,X_m\rangle \otimes_{(A,\Zbb)_{\square}}C)/(f_1,\ldots,f_n)$$
Therefore, it is enough to show that $$A\langle X_1,\ldots,X_m\rangle \otimes_{(A,\Zbb)_{\square}}C \simeq C\langle X_1,\ldots,X_m\rangle.$$
We take rings of definition $A_0\subset A$ and $C_0 \subset C$ such that the image of $A_0$ under $A\to C$ is contained in $C_0$.
Then it suffices to show that 
\begin{align}\label{eq3}
A_0\langle X_1,\ldots,X_m\rangle \otimes_{(A_0,\Zbb)_{\square}}C_0 \to C_0\langle X_1,\ldots,X_m\rangle.
\end{align}
is an equivalence.
Since both sides are $\pi$-adically complete by \cite[Proposition 2.12.10]{Mann22}, it is enough to show (\ref{eq3}) is an equivalence after taking modulo $\pi$.
The map (\ref{eq3}) modulo $\pi$ is equivalent to
$$A_0/\pi[X_1,\ldots,X_m] \otimes_{(A_0/\pi,\Zbb)_{\square}} C_0/\pi \to C_0/\pi[X_1,\ldots,X_m],$$
which is an isomorphism.
\end{proof}
\begin{remark}
By the proof, we find that $\pi_0(B\otimes_{(A,\Zbb)_{\square}}^{\Lbb} C)$ is isomorphic to the usual completed tensor product $\pi_0B \hotimes_{\pi_0A} \pi_0(C)$.
\end{remark}

By the proof, we also get the following lemma.

\begin{lemma}\label{lemma:flat basechange}
	Let $B\leftarrow A \to C$ be a diagram of (static) affinoid $K$-algebras.
	Assume that $A\to C$ is flat in the sense of \cite[Definition 7.3]{Mikami22} (i.e., $A(\ast)\to C(\ast)$ is flat in the usual sense).
	Then $B\otimes_{(A,\Zbb)_{\square}}^{\Lbb} C$ is static.
\end{lemma}
\begin{proof}
	We may assume that $B=A\langle X_1,\ldots,X_m\rangle$ or $B=A/I$ for some closed ideal $I\subset A$.
	If $B=A\langle X_1,\ldots,X_m\rangle$, we have 
	$$A\langle X_1,\ldots,X_m\rangle\otimes_{(A,\Zbb)_{\square}}^{\Lbb} C\simeq C\langle X_1,\ldots,X_m\rangle$$ by the proof of Proposition \ref{prop:tensor prod of animated affinoid}, so we get the claim.
	Next, suppose that $B=A/I$.
	Since $A$ is a Noetherian Banach $\Qbb_p$-algebra, $\underline{B}$ admits a resolution by finite free $\underline{A}$-modules.
	Therefore, by \cite[Lemma 5.10]{And21}, we get 
	$$\underline{B}\otimes_{(A,\Zbb)_{\square}}^{\Lbb} \underline{C}\simeq (B(\ast)\otimes_{A(\ast)}^{\Lbb}C(\ast))\otimes_{C(\ast)}^{\Lbb}\underline{C}.$$
	Since $A(\ast)\to C(\ast)$ is flat, we find that $B\otimes_{(A,\Zbb)_{\square}}^{\Lbb} C$ is static by \cite[Theorem 5.9]{And21}.
\end{proof}

\begin{definition}
\begin{enumerate}
\item
\textit{An animated affinoid pair of weakly finite type over $K$} is a pair $(A,A^+)$ where $A$ is an animated affinoid $K$-algebra and $A^+$ is a ring of integral elements of $\pi_0(A)$\footnote{To be precise, $A^+$
  is considered a subring of the usual ring $\pi_0(A)(\ast)$ as in Definition \ref{def:ana}. However, to simplify the notation, we simply write $A_0\subset \pi_0(A)$.}.
\item
For animated affinoid pairs $(A,A^+)$, $(B,B^+)$ of weakly finite type over $K$, a morphism $f \colon (A,A^+)\to (B,B^+)$ is a morphism $f \colon A \to B$ of animated affinoid $K$-algebras such that the image of $A^+$ under the morphism $\pi_0f \colon \pi_0A \to \pi_0B$ is contained in $B^+$.
\item
For an animated affinoid pair $(A,A^+)$ of weakly finite type over $K$, we can define an analytic animated ring $(A,A^+)_{\square}$ by Definition \ref{def:ana}.
If $A^+$ is equal to the ring of power-bounded elements of $\pi_0A$, then we simply denote $(A,A^+)_{\square}$ by $A_{\square}$. 
\end{enumerate}
\end{definition}

\begin{remark}
For an animated affinoid pair $(A,A^+)$ of weakly finite type over $K$, $(\pi_0A, A^+)$ is an affinoid pair of weakly finite type over $(K,K^+)$ in the sense of \cite{Hub96}, where $K^+$ is the smallest ring of integral elements of $K$.
\end{remark}

\begin{proposition}\label{prop:cateq}
Let $\Ani\Aff_K$ denote the $\infty$-category of animated affinoid pairs of weakly finite type over $K$.
Then the functor $$\Ani\Aff_K \to \AnRing_{(K,\Zbb)_{\square}};\; (A,A^+) \mapsto (A,A^+)_{\square}$$ is fully faithful, where $\AnRing_{(K,\Zbb)_{\square}}$ is the $\infty$-category of analytic animated $(K,\Zbb)_{\square}$-algebras.
\end{proposition}
\begin{proof}
It is enough to show that for animated affinoid pairs $(A,A^+)$, $(B,B^+)$ of weakly finite type over $K$, a morphism $f \colon A \to B$ of animated affinoid $K$-algebras defines a morphism $(A,A^+)_{\square} \to (B,B^+)_{\square}$ of analytic animated rings if and only if $f$ defines a morphism $(A,A^+)\to (B,B^+)$ of animated affinoid pairs of weakly finite type over $K$.
By Remark \ref{rem:sta}, we get
\begin{align*}
&\text{The morphism $f$ defines a morphism $(A,A^+)_{\square} \to (B,B^+)_{\square}$.} \\
\iff &\text{Any object $M$ of $\Dcal((B,B^+)_{\square})$ is $(A,A^+)_{\square}$-complete.}\\
\iff &\text{Any static object $M\in \Dcal((B,B^+)_{\square})^{\heart}$ is $(\pi_0A,A^+)_{\square}$-complete.}\\
\iff &\text{Any object $M$ of $\Dcal((\pi_0B,B^+)_{\square})$ is $(\pi_0A,A^+)_{\square}$-complete.}\\
\iff &\text{The morphism $f$ defines a morphism $(\pi_0A,A^+)_{\square} \to (\pi_0B,B^+)_{\square}$,} 
\end{align*}
where we note that for an analytic animated ring $\Acal$, an object $M$ of $\Dcal(\underline{\Acal})$ is $\Acal$-complete if and only if $H^i(M)$ is $\Acal$-complete for every $i\in \Zbb$.
By \cite[Proposition 3.34]{And21}, the last condition is equivalent to $\pi_0f(A^+) \subset B^+$, which proves the proposition.
\end{proof}

\begin{remark}
By definition, we have $(\pi_0A,A^+)_{\square}=\pi_0((A,A^+)_{\square})$ for an animated affinoid pair $(A,A^+)$ of weakly finite type over $K$.
\end{remark}

\begin{proposition}
Let $f\colon (A,A^+)\to (B,B^+)$ be a morphism of animated affinoid pairs of weakly finite type over $K$.
Then $f_{\square} \colon (A,A^+)_{\square}\to (B,B^+)_{\square}$ is steady.
\end{proposition}
\begin{proof}
By \cite[Proposition 2.3.19]{Mann22}, we may assume $(A,A^+)_{\square}=(K,\Zbb)_{\square}$.
Since $B$ is nuclear as an object of $\Dcal((K,\Zbb)_{\square})$, the map $(K,\Zbb)_{\square} \to (B, \Zbb)_{\square}$ is steady by \cite[Proposition 13.14]{AG}.
The morphism $(B, \Zbb)_{\square}\to (B,B^+)_{\square}$ can be obtained as a base change of $g\colon (B(\ast), \Zbb)_{\square}\to (B(\ast),B^+)_{\square}$, so it is enough to show that $g$ is steady.
It follows from \cite[Proposition 2.9.7]{Mann22}.
\end{proof}

\begin{remark}
The condition (2) in Definition \ref{def:aff} is necessary for the above proposition to hold.
\end{remark}

\begin{definition}
Let $(B,B^+)\leftarrow (A,A^+) \to (C,C^+)$ be a diagram of animated affinoid pairs of weakly finite type over $K$.
We define the derived tensor product $(B,B^+)\otimes_{(A,A^+)}^{\Lbb} (C,C^+)$ of $(B,B^+)$ and $(C,C^+)$ over $(A,A^+)$ as a pair of the animated affinoid $K$-algebra $B\otimes_{(A,\Zbb)_{\square}}^{\Lbb} C$ 
and the smallest ring of integral elements of $\pi_0(B\otimes_{(A,\Zbb)_{\square}}^{\Lbb} C)$ which contains the image of $\pi_0(B^+\otimes C^+)$.
\end{definition}

\begin{proposition}\label{prop:afftensor}
Let $(B,B^+)\leftarrow (A,A^+) \to (C,C^+)$ be a diagram of animated affinoid pairs of weakly finite type over $K$.
Then the natural map $$(B,B^+)_{\square} \otimes_{(A,A^+)_{\square}}^{\Lbb} (C,C^+)_{\square}\to ((B,B^+)\otimes_{(A,A^+)}^{\Lbb} (C,C^+))_{\square}$$ is an equivalence.
\end{proposition}
\begin{proof}
Since $B\otimes_{(A,\Zbb)_{\square}}^{\Lbb} C$ is $(B\otimes_{(A,\Zbb)_{\square}}^{\Lbb} C)_{\square}$-complete, it is also $(A,A^+)_{\square}$, $(B,B^+)_{\square}$ and $(C,C^+)_{\square}$-complete.
Therefore, we get $B\otimes_{(A,\Zbb)_{\square}}^{\Lbb} C \simeq B\otimes_{(A,A^+)_{\square}}^{\Lbb} C$, which is $(B,B^+)_{\square}$ and $(C,C^+)_{\square}$-complete.
It is enough to show that for an object $M$ of $\Dcal(B\otimes_{(A,\Zbb)_{\square}}^{\Lbb} C)$, $M$ is $((B,B^+)\otimes_{(A,A^+)}^{\Lbb} (C,C^+))_{\square}$-complete if and only if it is $(B,B^+)_{\square} \otimes_{(A,A^+)_{\square}}^{\Lbb} (C,C^+)_{\square}$-complete (equivalently, $(B,B^+)_{\square}$ and $(C,C^+)_{\square}$-complete).
By the definition of the analytic animated ring structures and \cite[Proposition 3.32]{And21}, it suffices to show that the smallest ring of integral elements of $\pi_0(B\otimes_{(A,\Zbb)_{\square}}^{\Lbb} C)$ containing the image of the map $B^+\otimes C^+ \to \pi_0(B\otimes_{(A,\Zbb)_{\square}}^{\Lbb} C)$ is equal to the ring of integral elements of $(B,B^+)\otimes_{(A,A^+)}^{\Lbb} (C,C^+)$.
It follows from the definition of $(B,B^+)\otimes_{(A,A^+)}^{\Lbb} (C,C^+)$.
\end{proof}

We will introduce the notion of rational localizations.
First, we interpret usual rational localizations in terms of analytic rings.
\begin{lemma}\label{lem:staloc}
Let $A$ be an affinoid $K$-algebra.
Let $f_1,\ldots,f_n,g$ be elements of $A$ which generates the unit ideal of $A$.
We regard $A$ as a $\Zbb[T]$-algebra by $T\mapsto g$, and regard $A\otimes_{(\Zbb[T],\Zbb)_{\square}}^{\Lbb}(\Zbb[T,T^{-1}],\Zbb)_{\square}=A\otimes_{\Zbb[T]} \Zbb[T,T^{-1}]=A[g^{-1}]$ as a $\Zbb[U_1,\ldots,U_n]$-algebra by $U_i\mapsto f_i/g$.
Then $$(A \otimes_{(\Zbb[T],\Zbb)_{\square}}^{\Lbb}(\Zbb[T,T^{-1}],\Zbb)_{\square}) \otimes_{(\Zbb[U_1,\ldots,U_n],\Zbb)_{\square}}^{\Lbb} \Zbb[U_1,\ldots,U_n]_{\square}$$ is isomorphic to the rational localization $A\langle f_1/g, \ldots, f_n/g \rangle$ in the usual sense.
Moreover, for a ring of integral elements $A^+$ of $A$, we have an equivalence 
\begin{align*}
&((A,A^+)_{\square} \otimes_{(\Zbb[T],\Zbb)_{\square}}^{\Lbb}(\Zbb[T,T^{-1}],\Zbb)_{\square}) \otimes_{(\Zbb[U_1,\ldots,U_n],\Zbb)_{\square}}^{\Lbb} \Zbb[U_1,\ldots,U_n]_{\square}\\
\simeq&(A\langle f_1/g, \ldots, f_n/g \rangle, A^+\langle f_1/g, \ldots, f_n/g \rangle)_{\square}. 
\end{align*}
\end{lemma}
\begin{proof}
Noting that the Koszul complex corresponding to the elements $f_1-gU_1, \ldots, f_n-gU_n$ in $A\langle U_1,\ldots,U_n\rangle$ is quasi-isomorphic to $A\langle f_1/g, \ldots, f_n/g \rangle$ (\cite[Exercise 1.9.26]{AWS17Ked}), we can prove it by the same way as in \cite[Proposition 4.11]{And21}.
\end{proof}

\begin{proposition}\label{prop:loc}
Let $A$ be an animated affinoid $K$-algebra.
Let $f_1,\ldots,f_n,g$ be elements of $\pi_0A$ which generates the unit ideal of $\pi_0A$.
We regard $A$ as a condensed animated $\Zbb[T]$-algebra by $T\mapsto g$, and regard $A\otimes_{(\Zbb[T],\Zbb)_{\square}}^{\Lbb}(\Zbb[T,T^{-1}],\Zbb)_{\square}=A\otimes_{\Zbb[T]}^{\Lbb} \Zbb[T,T^{-1}]$ as a condensed animated $\Zbb[U_1,\ldots,U_n]$-algebra by $U_i\mapsto f_i/g$.
Then $$A\langle f_1/g,\ldots,f_n/g\rangle \coloneqq(A \otimes_{(\Zbb[T],\Zbb)_{\square}}^{\Lbb}(\Zbb[T,T^{-1}],\Zbb)_{\square}) \otimes_{(\Zbb[U_1,\ldots,U_n],\Zbb)_{\square}}^{\Lbb} \Zbb[U_1,\ldots,U_n]_{\square}$$ is an animated affinoid $K$-algebra such that 
$$\pi_0(A\langle f_1/g,\ldots,f_n/g\rangle) \simeq (\pi_0A)\langle f_1/g,\ldots,f_n/g\rangle.$$
Moreover, for a ring of integral elements $A^+$ of $\pi_0A$, we have an equivalence 
\begin{align*}
&((A,A^+)_{\square} \otimes_{(\Zbb[T],\Zbb)_{\square}}^{\Lbb}(\Zbb[T,T^{-1}],\Zbb)_{\square}) \otimes_{(\Zbb[U_1,\ldots,U_n],\Zbb)_{\square}}^{\Lbb} \Zbb[U_1,\ldots,U_n]_{\square}\\
\simeq&(A\langle f_1/g, \ldots, f_n/g \rangle, A^+\langle f_1/g, \ldots, f_n/g \rangle)_{\square}. 
\end{align*}
\end{proposition}
\begin{proof}
We have a functor
$$F\coloneqq (- \otimes_{(\Zbb[T],\Zbb)_{\square}}^{\Lbb}(\Zbb[T,T^{-1}],\Zbb)_{\square}) \otimes_{(\Zbb[U_1,\ldots,U_n],\Zbb)_{\square}}^{\Lbb} \Zbb[U_1,\ldots,U_n]_{\square}$$
from $\Dcal((A,\Zbb)_{\square})$ to $\Dcal((A,\Zbb)_{\square})$.
This functor is right $t$-exact, so we get $$\pi_0(A\langle f_1/g,\ldots,f_n/g\rangle)\simeq \pi_0(F(\pi_0A))\simeq (\pi_0A)\langle f_1/g,\ldots,f_n/g\rangle,$$ 
where the last equivalence follows from Lemma \ref{lem:staloc}.
Next, we will prove that $A\langle f_1/g,\ldots,f_n/g\rangle$ is nuclear as an object of $\Dcal((K,\Zbb)_{\square})$.
By Proposition \ref{prop:nucsta} and right $t$-exactness of $F$, it is enough to show that $F(\pi_iA)$ is nuclear as an object of $\Dcal((K,\Zbb)_{\square})$ for every non-negative integer $i$.
Since $\pi_iA$ is a $\pi_0A$-module and two morphisms $(\Zbb[T],\Zbb)_{\square}\to (\Zbb[T,T^{-1}],\Zbb)_{\square}$ and $(\Zbb[U_1,\ldots,U_n],\Zbb)_{\square}\to \Zbb[U_1,\ldots,U_n]_{\square}$ are steady, 
we get
\begin{align*}
&\phantom{{}\simeq {}}F(\pi_i A)\\
&\simeq  \pi_iA \otimes_{(\pi_0A,\Zbb)_{\square}}^{\Lbb}((\pi_0A,\Zbb)_{\square} \otimes_{(\Zbb[T],\Zbb)_{\square}}^{\Lbb}(\Zbb[T,T^{-1}],\Zbb)_{\square}) \otimes_{((\pi_0A,\Zbb)_{\square} \otimes_{(\Zbb[T],\Zbb)_{\square}}^{\Lbb}(\Zbb[T,T^{-1}],\Zbb)_{\square})} \\
&\phantom{{}\simeq {}}(((\pi_0A,\Zbb)_{\square} \otimes_{(\Zbb[T],\Zbb)_{\square}}^{\Lbb}(\Zbb[T,T^{-1}],\Zbb)_{\square})\otimes_{(\Zbb[U_1,\ldots,U_n],\Zbb)_{\square}}^{\Lbb}\Zbb[U_1,\ldots,U_n]_{\square})\\
&\simeq \pi_iA \otimes_{(\pi_0A,\Zbb)_{\square}}^{\Lbb} ((\pi_0A)\langle f_1/g, \ldots, f_n/g \rangle, \Zbb[f_1/g, \ldots, f_n/g])_{\square}.
\end{align*}
By Lemma \ref{lem:nuccoef} and Proposition \ref{prop:nucsta}, $\pi_i A$ is nuclear as an object of $\Dcal((\pi_0A,\Zbb)_{\square})$, so $F(\pi_i A)$ is also nuclear as an object of $\Dcal(((\pi_0A)\langle f_1/g, \ldots, f_n/g \rangle, \Zbb[f_1/g, \ldots, f_n/g])_{\square})$.
Therefore, $F(\pi_i A)$ is nuclear as an object of $\Dcal((K,\Zbb)_{\square})$ by Corollary \ref{cor:nuc} and Lemma \ref{lem:nuccoef}.
Finally, the last claim follows from \cite[Proposition 3.32]{And21}.
\end{proof}

\begin{definition}
Let $(A,A^+)$ be an animated affinoid pair of weakly finite type over $K$.
A \textit{rational localization} of $(A,A^+)$ is a morphism from $(A,A^+)$ of the form $(A,A^+)\to(A\langle f_1/g, \ldots, f_n/g \rangle, A^+\langle f_1/g, \ldots, f_n/g \rangle)$ as in Proposition \ref{prop:loc}.
We let $\Aff\Op_{(A,A^+)}$ denote the $\infty$-category of rational localizations of $(A,A^+)$.
\end{definition}

\begin{lemma}
Let $(A,A^+)$ be an animated affinoid pair of weakly finite type over $K$, and $f\colon (A,A^+) \to (B,B^+)$ be a rational localization.
Then $f_{\square} \colon (A,A^+)_{\square} \to (B,B^+)_{\square}$ is a steady localization in the sense of \cite[Definition 12.16]{AG}.
\end{lemma}
\begin{proof}
It follows from the fact that $$(\Zbb[T],\Zbb)_{\square} \to \Zbb[T,T^{-1}],\Zbb)_{\square}$$ and $$(\Zbb[U_1,\ldots,U_n],\Zbb)_{\square}\to \Zbb[U_1,\ldots,U_n]_{\square}$$ are steady localizations by \cite[Lemma 3.8]{And21}.
\end{proof}

\begin{lemma}\label{lem:lff}
Let $(A,A^+), (C, C^+)$ be animated affinoid pairs of weakly finite type over $K$, and $f\colon (A,A^+) \to (B,B^+)$ be a rational localization.
Then the natural morphism $$\Map_{(A,A^+)/}((B,B^+), (C,C^+))\to \Map_{(\pi_0A,A^+)/}((\pi_0B,B^+), (\pi_0C,C^+))$$ is an equivalence of anima.
\end{lemma}
\begin{proof}
By Proposition \ref{prop:cateq}, it is enough to show that the natural morphism 
$$\Map_{(A,A^+)_{\square}/}((B,B^+)_{\square}, (C,C^+)_{\square})\to \Map_{(\pi_0A,A^+)_{\square}/}((\pi_0B,B^+)_{\square}, (\pi_0C,C^+)_{\square})$$ 
is an equivalence of anima.
Since $(A,A^+)_{\square} \to (B, B^+)_{\square}$ and $(\pi_0A,A^+)_{\square} \to (\pi_0B, B^+)_{\square}$ are localizations, the above mapping anima are either empty or contractible.
Therefore, it is enough to show that if $\Map_{(\pi_0A,A^+)_{\square}/}((\pi_0B,B^+)_{\square}, (\pi_0C,C^+)_{\square})\neq \emptyset$, then $\Map_{(A,A^+)_{\square}/}((B,B^+)_{\square}, (C,C^+)_{\square})\neq \emptyset$.
We put $$(B,B^+)=(A\langle f_1/g, \ldots, f_n/g \rangle, A^+\langle f_1/g, \ldots, f_n/g \rangle).$$
Since the image of $g$ in $\pi_0C$ is invertible, we get $(\Zbb[T,T^{-1}],\Zbb)_{\square} \to (C, C^+)_{\square};\; T \mapsto g$.
Moreover the image of $f_1/g, \ldots, f_n/g$ in $\pi_0C$ lies in $C^+$, we get $\Zbb[U_1, \ldots U_n]_{\square} \to (C, C^+)_{\square} ;\; U_i \mapsto f_i/g$.
Therefore, we get 
\begin{align*}
(B,B^+)_{\square} &\simeq ((A,A^+)_{\square} \otimes_{(\Zbb[T],\Zbb)_{\square}}^{\Lbb}(\Zbb[T,T^{-1}],\Zbb)_{\square}) \otimes_{(\Zbb[U_1,\ldots,U_n],\Zbb)_{\square}}^{\Lbb} \Zbb[U_1,\ldots,U_n]_{\square}\\
& \to (C,C^+)_{\square},
\end{align*}
which proves the lemma.
\end{proof}

\begin{proposition}\label{prop:ratopcateq}
Let $(A,A^+)$ be an animated affinoid pair of weakly finite type over $K$.
Then the functor $\Aff\Op_{(A,A^+)} \to \Aff\Op_{(\pi_0A,A^+)} ;\; (B,B^+) \to (\pi_0B, B^+)$, which is well-defined by Proposition \ref{prop:loc}, is a categorical equivalence.
\end{proposition}
\begin{proof}
This functor is obviously essentially surjective.
Moreover, it is fully faithful by Lemma \ref{lem:lff}.
\end{proof}

\begin{proposition}\label{prop:ratcov}
Let $(A,A^+)$ be an animated affinoid pair of weakly finite type over $K$.
Let $\{ (A,A^+) \to (A_i,A_i^+)\}_{i=1}^n$ be a finite family of rational localizations.
If $\{ (\pi_0A, A^+) \to (\pi_0A_i,A_i^+)\}_{i=1}^n$ is a rational covering (i.e., $\{ \Spa(\pi_0A_i, A_i^+) \to \Spa(\pi_0A,A^+)\}_{i=1}^n$ is a rational open covering),
then $\Dcal((A,A^+)_{\square}) \to \prod_{i=1}^n \Dcal((A_i,A_i^+)_{\square})$ is conservative.
We will call such a family $\{ (A,A^+) \to (A_i,A_i^+)\}_{i=1}^n$ a \textit{rational covering}.
\end{proposition}
\begin{proof}
It can be proved by the same way as in \cite[Chapter 4]{And21}.
\end{proof}

\begin{corollary}\label{cor:ratcov}
We endow $\Aff\Op_{(A,A^+)}$ with the Grothendieck topology generated by rational coverings.
Then the presheaf
$$\Aff\Op_{(A,A^+)}^{\op} \to \Pr^L ;\; (B,B^+) \to \Dcal((B,B^+)_{\square})$$
is a sheaf of presentable $\infty$-categories.
\end{corollary}
\begin{proof}
It follows from Proposition \ref{prop:ratcov} and \cite[Proposition 12.18]{AG}.
\end{proof}

Next, we will introduce the notion of faithfully flat morphisms for morphisms of animated affinoid $K$-algebras.

\begin{definition}
A morphism $A\to B$ of animated affinoid $K$-algebras is \textit{faithfully flat} if the morphism $\pi_0(A) \to \pi_0(B)$ of affinoid $K$-algebras is faithfully flat (in the sense of \cite[Definition 4.3]{Mikami22}) and $B_{\square} \otimes_{A_{\square}}^{\Lbb} \pi_0A_{\square}$ is equivalent to $\pi_0B_{\square}$.
\end{definition}

\begin{remark}
Let $f \colon A \to B$ be a morphism of discrete animated ring. 
Then $f$ is faithfully flat if and only if $\pi_0f \colon \pi_0A \to \pi_0B$ is faithfully flat and $B \otimes_{A}^{\Lbb} \pi_0A$ is equivalent to $\pi_0B$ (\cite[Lemma 5.2.2]{CS24}).
Therefore, the above definition is natural.
\end{remark}

The following is an analog of \cite[Lemma 4.4]{Mikami22}. 
\begin{lemma}\label{lem:flattening}
Let $A\to B$ a faithfully flat morphism of animated affinoid $K$-algebras.
Then there exists a rational covering $\{ (A,(\pi_0A)^{\circ}) \to (A_i,(\pi_0A_i)^{\circ})\}_{i=1}^n$ which satisfies that for each $i$ there exists a faithfully flat map $A_i^{\prime} \to B_i^{\prime}$ of admissible $\Ocal_K$-algebras (\cite[Definition 4.1, Definition 4.3]{Mikami22}) such that $A_i^{\prime}[1/\pi] \to B_i^{\prime}[1/\pi]$ is isomorphic to $$\pi_0A_i \to \pi_0B \otimes_{\pi_0A_{\square}}^{\Lbb} (\pi_0A_i)_{\square}.$$
\end{lemma}
\begin{proof}
It follows \cite[Lemma 4.4]{Mikami22} and Proposition \ref{prop:ratopcateq}.
\end{proof}


\section{Boundedness of $N_{B/A}$ for a faithfully flat morphism $A\to B$ of affinoid $K$-algebras}
In this section, we will prove an analog of \cite[Proposition 2.20]{Mikami22}, which will play an important role in the proof of the main theorem.

\begin{notation}
For a $t$-adically complete condensed animated ring $A$ such that $A/^{\Lbb}t$ is discrete ($t\in \pi_0(A)$), we denote $(A,\pi_0A)_{\square}$ by $A_{\square}$.
\end{notation}

Let $K$ be a complete non-archimedean field, and let $\Ocal_{K}$ denote the ring of integers of $K$, and $\pi \in K$ denote a pseudo-uniformizer of $K$.
Let $A \to B$ be a faithfully flat morphism of affinoid $K$-algebras.
We assume that there exists a faithfully flat map $A_0 \to B_0$ of admissible $\Ocal_K$-algebras such that $A_0[1/\pi] \to B_0[1/\pi]$ is isomorphic to $A\to B$.
First, we recall the following theorem proved in \cite{Mikami22}\footnote{Although this theorem is proved in the setting of ``non-light'' condensed mathematics in \cite{Mikami22}, the same arguments apply to the ``light'' context.}.
\begin{theorem}[{\cite[Theorem 4.10]{Mikami22}}]\label{thm:N}
There exists an object $N_{B_0/A_0} \in \Dcal(\underline{B_0})$ satisfying the following conditions:
\begin{itemize}
\item
The object $N_{B_0/A_0}$ is ${A_0}_{\square}$-complete and compact as an object of $\Dcal({A_0}_{\square})$.
\item
There exists an equivalence of functors from $\Dcal({A_0}_{\square})$ to $\Dcal({B_0}_{\square})$
$$-\otimes_{{A_0}_{\square}}^{\Lbb} {B_0}_{\square} \simeq R\intHom_{A_0}(N_{B_0/A_0},-).$$
\end{itemize}
We denote $N_{B_0/A_0} \otimes_{B_0}^{\Lbb} B$ by $N_{B/A}$.
Then $N_{B/A} \in \Dcal(\underline{B})$ satisfies the following conditions:
\begin{itemize}
\item
The object $N_{B/A}$ is $A_{\square}$-complete and compact as an object of $\Dcal(A_{\square})$.
\item
There exists an equivalence of functors from $\Dcal(A_{\square})$ to $\Dcal(B_{\square})$
$$-\otimes_{A_{\square}}^{\Lbb} B_{\square} \simeq R\intHom_{A}(N_{B/A},-).$$
\end{itemize}
\end{theorem}

We denote the $n$-fold derived tensor product of $N_{B/A}$ (resp. $N_{B_0/A_0}$) over $A_{\square}$ (resp. ${A_0}_{\square}$) by $N_{B/A}^{\otimes n}$ (resp. $N_{B_0/A_0}^{\otimes n}$).

\begin{theorem}\label{thm:affbdd}
For every $n\geq 1$, $N_{B/A}^{\otimes n}$ is quasi-isomorphic to a complex of the form $0 \to M^0 \to M^1 \to M^2 \to 0$ where $M^0$ is placed in cohomological degree $0$ and $M^0, M^1,M^2$ are projective objects of $\Dcal(A_{\square})^{\heart}$.
In particular, $-\otimes_{A_{\square}}^{\Lbb} B_{\square}$ has finite Tor-dimension $\leq 2$.
\end{theorem}
\begin{proof}
We will show that $N_{B_0/A_0}^{\otimes n}$ is quasi-isomorphic to a complex of the form $0 \to M^0_0 \to M^1_0 \to M^2_0 \to 0$ where $M^0_0$ is placed in cohomological degree $0$ and $M^0_0, M^1_0,M^2_0$ are projective objects of $\Dcal((A_0)_{\square})^{\heart}$.
Since $N_{B_0/A_0}^{\otimes n}$ is equivalent to $N_{B_0^{\otimes n}/A_0}$ by \cite[Corollary 4.8]{Mikami22}, where $B_0^{\otimes n}$ is the classical $n$-fold completed tensor product of $B_0$ over $A_0$, we may assume $n=1$. 
It is enough to show that for every static object $M$ of $\Dcal({A_0}_{\square})^{\heart}$, $R\intHom_{A_0}(N_{B_0/A_0}, M) \simeq M\otimes_{{A_0}_{\square}}^{\Lbb} {B_0}_{\square}$ lies in $\Dcal({A_0}_{\square})^{[-2,0]}$.
Since we have $M\otimes_{{A_0}_{\square}}^{\Lbb}{B_0}_{\square}\in\Dcal({A_0}_{\square})^{\leq0}$, it suffices to show $R\intHom_{A_0}(N_{B_0/A_0}, M) \in \Dcal({A_0}_{\square})^{\geq -2}$.
By the proof of \cite[Theorem 4.13]{Mikami22}, we can take a small $\pi$-adic ring $A^{\prime}$ with a morphism $A^{\prime} \to A_0$ and a compact object $N^{\prime}\in\Dcal(A^{\prime}_{\square})$ satisfying the following conditions:
\begin{enumerate}
\item
In $A^{\prime}$, $\pi$ is a non-zero-divisor.
\item 
There is an equivalence $N^{\prime}\otimes_{A^{\prime}_{\square}}^{\Lbb} {A_0}_{\square} \simeq N_{B_0/A_0}$.
\item\label{2}
The $(A^{\prime}/\pi)_{\square}$-module $N^{\prime}/^{\Lbb}\pi$ has projective-amplitude in $[0,1]$.
\end{enumerate}
We will prove $R\intHom_{A^{\prime}}(N^{\prime}, M) \in \Dcal(A^{\prime}_{\square})^{\geq -2}$ for every static object $M$ of $\Dcal(A^{\prime}_{\square})^{\heart}$.
Since $A^{\prime}$ is small, every compact object in $\Dcal(A^{\prime}_{\square})$ is $\pi$-adically complete (\cite[Proposition 3.8]{Mikami22}).
We write $M$ as a filtered colimit $\varinjlim_{\lambda} M_{\lambda}$ of $\pi$-adically complete objects of $\Dcal(A^{\prime}_{\square})$.
Then $H^0(M_{\lambda})$ is also $\pi$-adically complete (\cite[Proposition 3.2]{Mikami22}), so we can write $M$ as the filtered colimit  $\varinjlim_{\lambda} H^0(M_{\lambda})$ of static $\pi$-adically complete objects. 
Therefore, we may assume that $M$ is $\pi$-adically complete.
Then $R\intHom_{A^{\prime}}(N^{\prime}, M)$ is $\pi$-adically complete, and  we have an equivalence $$R\intHom_{A^{\prime}}(N^{\prime}, M)/^{\Lbb} \pi \simeq R\intHom_{A^{\prime}/\pi}(N^{\prime}/^{\Lbb}\pi, M/^{\Lbb}\pi).$$
Since the functor $R\varprojlim_{n\in \Nbb}$ is left $t$-exact, it is enough to show $$R\intHom_{A^{\prime}/\pi}(N^{\prime}/^{\Lbb}\pi, M/^{\Lbb}\pi) \in \Dcal((A^{\prime}/\pi)_{\square})^{\geq -2}.$$
It follows from (\ref{2}).
\end{proof}


\section{Fppf-descent for condensed animated rings}

First, we will compare the notion of descendability in \cite{Mikami22} and \cite{Mann22}.
We roughly recall the notion of descendability in \cite{Mann22}. For details, see \cite[2.5, 2.6]{Mann22}.
\begin{definition}[{\cite[Definition 2.5.1]{Mann22}}]
Let $\Acal$ be an analytic animated ring.
Let $\End(\Dcal(\Acal))$ denote the $\infty$-category of $\Dcal(\Acal)$-enriched exact endofunctors of $\Dcal(\Acal)$.
It is a stable $\infty$-category and it has the composition monoidal structure.
\end{definition}


\begin{remark}\label{rem:ff}
The functor $\Dcal(\Acal)^{\op} \to \End(\Dcal(\Acal)) ;\; N \mapsto R\intHom_{\Acal}(N,-)$ is monoidal, exact, and fully faithful (\cite[Remark 2.5.15]{Mann22}).
\end{remark}

\begin{definition}[{\cite[Definition 2.6.7]{Mann22}}]
For $n$ be a non-negative integer, a steady morphism $f\colon \Acal\to\Bcal$ of analytic animated rings is \textit{descendable} of index $\leq n$ if the natural morphism $\Kcal_f^n \to \id$ is zero, where $\Kcal_f \coloneqq \fib(\id \to (-\otimes_{\Acal}\Bcal))$.
A steady morphism $f\colon \Acal\to\Bcal$ of analytic animated rings is \textit{descendable} if it is descendable of index $\leq n$ for some non-negative integer $n$.
\end{definition}

\begin{remark}
When we define the descendability for morphisms of analytic animated rings, it is not necessary to introduce $\infty$-categories $\End(\Dcal(\Acal))^{\std}$ and $\Ecal(\Acal)$ (\cite[Definition 2.5.8, Definition 2.5.17]{Mann22}).
These $\infty$-categories are necessary to define the descendability for morphisms of analytic spaces, because we cannot glue the $\infty$-categories $\End(\Dcal(\Acal))$ along analytic coverings.
\end{remark}

Let $f \colon \Acal \to \Bcal$ be a steady morphism of analytic animated rings.
We assume that there exists a compact object $N \in \Dcal(\Acal)$ such that there is an equivalence of functors from $\Dcal(\Acal)$ to $\Dcal(\Acal)$
$$-\otimes_{\Acal}^{\Lbb} \Bcal \simeq R\intHom_{\Acal}(N,-).$$
We inductively define full $\infty$-subcategories $\Ccal(N,m)$ of $\Dcal(\Acal)$ for positive integers $m$ as follows:
\begin{itemize}
\item
The full $\infty$-subcategory $\Ccal(N,1)$ is the smallest full $\infty$-subcategory which contains $N$ and is closed under tensor products, retracts, and shifts.
\item
For $m\geq 2$, the full $\infty$-subcategory $\Ccal(N,m)$ is the smallest full $\infty$-subcategory which contains $\Ccal(N,m-1)$ together with $\fib(M\to M^{\prime})$ for any $M\in \Ccal(N,i)$ and $M^{\prime}\in \Ccal(N,j)$ with $i+j=m$ and is closed under tensor products, retracts, and shifts.
\end{itemize}

For $M\in \Dcal(\Acal)$ and $m\geq 0$, we write $M^{\otimes m}$ for the $m$-fold derived tensor product of $M$ in $\Dcal(\Acal)$.

\begin{lemma}\label{lem:desiff}
If $\Ccal(N,n)$ contains $\underline{\Acal}$, then $f$ is descendable of index $\leq n$.
Conversely, if $f$ is descendable, then there exists a positive integer $m$ such that $\underline{\Acal} \in \Ccal(N,m)$.
\end{lemma}
\begin{proof}
First, we assume that $\underline{\Acal}\in \Ccal(N,n)$.
By the same argument in \cite[The proof of $(i)\Rightarrow(ii)$ in Proposition 2.6.5]{Mann22}, we can show that for $M\in \Ccal(N,m)$ and composable morphisms $f_1,\ldots,f_m$, if $N\otimes_{\Acal}^{\Lbb} f_i$ is zero for each $i$, then $M\otimes_{\Acal}^{\Lbb}(f_1\circ \ldots \circ f_m)$ is also zero.
In particular, for composable morphisms $f_1,\ldots,f_n$, if $N\otimes_{\Acal}^{\Lbb} f_i$ is zero for each $i$, then the composition $f_1\circ \ldots \circ f_n$ is also zero.
Let $L=\cofib(N\to \underline{\Acal})$.
We have $\Kcal_f\simeq R\intHom_{\Acal}(L,-).$
Then $f$ is descendable of index $\leq m$ if and only if $\underline{\Acal}\to L^{\otimes n}$ is zero.
The natural morphism $-\otimes_{\Acal}^{\Lbb}\Bcal \to (-\otimes_{\Acal}^{\Lbb}\Bcal)\otimes_{\Acal}^{\Lbb}\Bcal$ admits a retraction.
In other words, the natural morphism $R\intHom_{\Acal}(N,-)\to R\intHom_{\Acal}(N^{\otimes 2},-)$ admits a retraction.
Therefore, the natural morphism $N^{\otimes 2}\to N$ admits a section.
Since there is a fiber sequence 
$$N^{\otimes 2}\to N \to N\otimes_{\Acal}^{\Lbb}L,$$
we find that $N \to N\otimes_{\Acal}^{\Lbb}L$ is zero.
Similarly, for any $i\geq 0$, $N\otimes_{\Acal}^{\Lbb}L^{\otimes i} \to N\otimes_{\Acal}^{\Lbb}L^{\otimes i+1}$ is also zero.
Therefore, $\underline{\Acal}\to L\to \cdots \to L^{\otimes n}$ is also zero.

Next, we assume that $f$ is descendable of index $\leq n$.
Then the morphism $\underline{\Acal}\to L^{\otimes n}$ is zero.
We set $\Ccal(N)=\bigcup_m\Ccal(N,m)$, which is closed under tensor products, retracts, shifts, and fibers.
Let us prove $\underline{\Acal} \in \Ccal(N)$. 
Since $\Acal$ is a retract of $\fib(\underline{\Acal}\to L^{\otimes n})$, it is enough to show that $\fib(\underline{\Acal}\to L^{\otimes n}) \in \Ccal(N)$.
We proceed by induction on $n$.
The case of $n=1$ is clear.
We assume that $\fib(\underline{\Acal}\to L^{\otimes n-1}) \in \Ccal(N)$.
We have a fiber sequence
$$N\otimes_{\Acal}^{\Lbb}\fib(\underline{\Acal}\to L^{\otimes n-1})\to \fib(\underline{\Acal}\to L^{\otimes n-1}) \to \fib(L\to L^{\otimes n}).$$
By the induction hypothesis, we get $N\otimes_{\Acal}^{\Lbb}\fib(\underline{\Acal}\to L^{\otimes n-1}), \fib(\underline{\Acal}\to L^{\otimes n-1})\in \Ccal(N)$.
Therefore, we get $\fib(L\to L^{\otimes n})\in \Ccal(N)$.
Moreover, we have a fiber sequence
$$N=\fib(\underline{\Acal}\to L)\to \fib(\underline{\Acal}\to L^{\otimes n})\to \fib(L\to L^{\otimes n}).$$
Since $N,\fib(L\to L^{\otimes n})\in \Ccal(N)$, we get $\fib(\underline{\Acal}\to L^{\otimes n}) \in \Ccal(N)$.
\end{proof}

We have a natural morphism $N\to \underline{\Acal}$ which corresponds to $$\id =R\intHom_{\Acal}(\underline{\Acal},- )\Rightarrow (-\otimes_{\Acal}^{\Lbb} \Bcal) \simeq R\intHom_{\Acal}(N,- ) \colon \Dcal(\Acal)\to \Dcal(\Acal).$$
From this, we get an augmented semisimplicial object $N^{\bullet}=\{N^{\otimes m+1}\}_{n\geq -1}$ and a natural morphism $\varinjlim_{[m] \in \Delta_{s}^{\op}}N^{\otimes(m+1)} \to N^{\otimes 0}=\underline{\Acal}$.
\begin{proposition}\label{prop:desbdd}
Let $n$ be a positive integer.
We assume that for every positive integer $r$, $N^{\otimes r}$ lies in $\Dcal(\Acal)^{\leq n}$.
If the natural map $\varinjlim_{[m] \in \Delta_{s}^{\op}}N^{\otimes(m+1)} \to \underline{\Acal}$ is an equivalence, then $f$ is descendable of index $\leq n+1$.
\end{proposition}
\begin{proof}
We will prove $\underline{\Acal} \in \Ccal(N,n+1)$.
We put $N_t\coloneqq \varinjlim_{[m] \in \Delta_{s,\leq t}^{\op}}N^{\otimes(m+1)}$ and $F_t\coloneqq \cofib(N_t\to\underline{\Acal})$ for every non-negative integer $t$.
By \cite[Lemma 4.12]{Mikami22}, we have a fiber sequence 
$$N_{t} \to N_{t+1} \to N^{\otimes(t+2)}[t+1].$$
Since $N^{\otimes r}$ lies in $\Dcal(\Acal)^{\leq n}$, we get that $\tau^{\geq n-t+1}(N_t) \to \tau^{\geq n-t+1}(N_{t+1})$ is an equivalence and $H^{n-t}(N_t) \to H^{n-t}(N_{t+1})$ is surjective.
For $k\in \Zbb$, we have an equivalence $\varinjlim_{m\geq 0}\tau^{\geq k}N_m \simeq \tau^{\geq k}\underline{\Acal}$, so we find that $\tau^{\geq 1} N_n \to \tau^{\geq 1}\underline{\Acal}$ is an equivalence and $H^0(N_n)\to H^0(\underline{\Acal})$ is surjective.
Therefore, we get that $F_n$ lies in $\Dcal(\Acal)^{\leq -1}$.
We have the following fiber sequences:
$$\xymatrix{
N_n \ar[r] &\underline{\Acal} \ar[r] &F_n  \\
N_{0} \ar[r]\ar[u] &\underline{\Acal} \ar[r]_{\alpha}\ar[u] &F_{0}\ar[u]_{\beta}.
}
$$
Since $F_n$ lies in $\Dcal(\Acal)^{\leq -1}$, $\beta\circ\alpha$ is zero.
Therefore, we find that $\underline{\Acal}$ is a retract of $N_n$.
We can easily prove that $N_n$ lies in $\Ccal(N,n+1)$ by using \cite[Lemma 4.12]{Mikami22}, which proves $\underline{\Acal} \in \Ccal(N,n+1)$.
\end{proof}

We will also need the following lemma.
\begin{lemma}\label{lem:desconv}
We assume that $f$ is descendable.
Then the natural map $$\varinjlim_{[m] \in \Delta_{s}^{\op}}N^{\otimes(m+1)} \to \underline{\Acal}$$ is an equivalence.
\end{lemma}
\begin{proof}
Let $\Ccal$ be the full $\infty$-subcategory of $\Dcal(\Acal)$ consisting of those objects $M$ such that $\varinjlim_{[m] \in \Delta_{s}^{\op}}(N^{\otimes(m+1)} \otimes_{\Acal}^{\Lbb} M) \to M$ is an equivalence.
We want to show $\underline{\Acal} \in \Ccal$.
Since $\Ccal$ is stable under finite colimits, shifts, tensor products, and retracts, it is enough to show $N\in \Ccal$ by Lemma \ref{lem:desiff}. 
We have an augmented cosimplicial analytic animated ring $\{\Bcal^{n+1} \}_{n\geq -1}$, where $\Bcal^{n+1}$ is the $(n+1)$-fold derived tensor product of $\Bcal$ over $\Acal$. 
Since $\Acal \to \Bcal$ is steady, we have an equivalence of functors from $\Dcal(\Acal)$ to $\Dcal(\Acal)$
$$-\otimes_{\Acal}^{\Lbb} \Bcal^{n+1} \simeq R\intHom_{\Acal}(N^{\otimes(n+1)},-).$$
Therefore, we can extend the augmented semisimplicial object $N^{\bullet}$ in $\Dcal(\Acal)$ to an augmented simplicial object $\{N^{\otimes(n+1)}\}_{n\geq -1}$.
Since an augmented simplicial object $\{N^{\otimes(n+1)}\otimes_{\Acal}^{\Lbb} N\}_{n\geq -1}$ is split (\cite[Definition 4.7.2.2]{HA}), we find that $$\varinjlim_{[m] \in \Delta^{\op}}(N^{\otimes(m+1)} \otimes_{\Acal}^{\Lbb} N) \to N$$ is an equivalence.
By \cite[Lemma 6.5.3.7]{HTT}, $\varinjlim_{[m] \in \Delta_s^{\op}}(N^{\otimes(m+1)} \otimes_{\Acal}^{\Lbb} N) \to N$ is also an equivalence, so $N$ lies in $\Ccal$. 

\end{proof}

The following lemmas are used to associate the properties of the scalar extension functor along a map of analytic animated rings with those along a map of analytic rings.

\begin{lemma}\label{lem:limtor}
Let $f \colon \Acal \to \Bcal$ be a morphism of analytic animated $\Zbb_{\square}$-algebras such that $\pi_0(\Acal)\otimes_{\Acal}^{\Lbb} \Bcal \simeq \pi_0(\Bcal)$. 
If the functor $-\otimes_{\pi_0(\Acal)}^{\Lbb} \pi_0(\Bcal) \colon \Dcal(\pi_0(\Acal))\to \Dcal(\pi_0(\Bcal))$ preserves all small limits and has finite Tor-dimension $\leq n$ for some $n\in \Zbb_{\geq0}$, 
then the functor $-\otimes_{\Acal}^{\Lbb} \Bcal \colon \Dcal(\Acal) \to \Dcal(\Bcal)$ also preserves all small limits and has finite Tor-dimension $\leq n$
\end{lemma}
\begin{proof}
A static object $M$ of $\Dcal(\Acal)$ can be regarded as a $\pi_0\Acal$-module and we have an equivalence $M\otimes_{\Acal}^{\Lbb} \Bcal \simeq M\otimes_{\pi_0\Acal}^{\Lbb} \pi_0\Bcal$, so $-\otimes_{\Acal}^{\Lbb} \Bcal \colon \Dcal(\Acal) \to \Dcal(\Bcal)$ has finite Tor-dimension $\leq n$.
Let us prove that $-\otimes_{\Acal}^{\Lbb} \Bcal \colon \Dcal(\Acal) \to \Dcal(\Bcal)$ preserves all small limits.
It is enough to show that it preserves all direct products.
Let $\{M_{\lambda}\}_{\lambda \in \Lambda}$ be a family of objects in $\Dcal(\Acal)$.
We will show that $H^r((\prod_{\lambda}M_{\lambda})\otimes_{\Acal}^{\Lbb} \Bcal) \to H^r(\prod_{\lambda}(M_{\lambda}\otimes_{\Acal}^{\Lbb} \Bcal))$ is an isomorphism for every integer $r$.
By taking a shift, we may assume $r=0$.
Note that the direct product functor in $\Dcal(\Zbb_{\square})$ is $t$-exact because the functor $R\Gamma(S,-)\colon \Dcal(\Zbb_{\square}) \to \Dcal(\Ab)$ is $t$-exact for any light profinite set $S$ by \cite[Theorem 3.2.4]{RC26}.
Therefore, the direct product functor in $\Dcal(\Acal)$ is also $t$-exact.
Moreover $-\otimes_{\Acal}^{\Lbb} \Bcal \colon \Dcal(\Acal) \to \Dcal(\Bcal)$ has finite Tor-dimension, so we may assume that $\{M_{\lambda}\}_{\lambda}$ is uniformly bounded, and then we may assume that $M_{\lambda}$ is static for any $\lambda$.
Then the claim follows from the fact that $-\otimes_{\pi_0(\Acal)}^{\Lbb} \pi_0(\Bcal) \colon \Dcal(\pi_0(\Acal)) \to \Dcal(\pi_0(\Bcal))$ preserves direct products.
\end{proof}

\begin{lemma}\label{lem:bddeq}
Let $\Acal$ be an analytic animated ring.
Let $f \colon M\to N$ be a morphism of bounded above objects of $\Dcal(\Acal)$.
If $f \otimes_{\Acal}^{\Lbb} \pi_0\Acal \colon M\otimes_{\Acal}^{\Lbb} \pi_0\Acal \to N\otimes_{\Acal}^{\Lbb} \pi_0\Acal$ is an equivalence, then $f$ is also an equivalence.
\end{lemma}
\begin{proof}
We may assume that $M,N \in \Dcal(\Acal)^{\leq0}$.
By induction, it is enough to show that $H^0(M) \to H^0(N)$ is an equivalence.
It follows from $H^0(M) \simeq H^0(M\otimes_{\Acal}^{\Lbb} \pi_0\Acal)$ and $H^0(N) \simeq H^0(N\otimes_{\Acal}^{\Lbb} \pi_0\Acal)$
\end{proof}

Let $K$ be a complete non-archimedean field.
Let $f \colon A\to B$ be a faithfully flat morphism of discrete animated rings such that $\pi_0f \colon \pi_0A \to \pi_0B$ is finitely presented (resp. a faithfully flat morphism of animated affinoid $K$-algebras such that $\pi_0f \colon \pi_0A \to \pi_0B$ is isomorphic to $A_0[1/\pi] \to B_0[1/\pi]$ for a faithfully flat map $A_0 \to B_0$ of admissible $\Ocal_K$-algebras).
We will prove that $f_{\square} \colon A_{\square} \to B_{\square}$ is descendable of index $\leq2$ (resp. $\leq 3$). 
First, we will prove that there exists a compact object $N_{B/A} \in \Dcal(\Acal)$ such that there is an equivalence $$-\otimes_{A_{\square}}^{\Lbb}B_{\square} \simeq R\intHom_A(N_{B/A},-)$$ of functors from $\Dcal(\Acal)$ to $\Dcal(\Acal)$ (Theorem \ref{existN}).
We take $N_{\pi_0B/\pi_0A}\in \Dcal(\pi_0A_{\square})$ and an equivalence of functors from $\Dcal(\pi_0A_{\square})$ to $\Dcal(\pi_0B_{\square})$
$$-\otimes_{\pi_0A_{\square}}^{\Lbb} \pi_0B_{\square} \simeq R\intHom_{\pi_0A}(N_{\pi_0B/\pi_0A},-)$$ (\cite[Corollary 2.11]{Mikami22}, Theorem \ref{thm:N}).

\begin{lemma}\label{lem:lim}
The functor $$-\otimes_{A_{\square}}^{\Lbb} B_{\square} \colon \Dcal(A_{\square}) \to \Dcal(B_{\square})$$ preserves all small limits and has finite Tor-dimension $\leq 1$ (resp. $\leq 2$).
\end{lemma}
\begin{proof}
By Lemma \ref{lem:limtor}, we may assume that $A$ and $B$ are static.
The claim about finite Tor-dimension follows from \cite[Proposition 2.20]{Mikami22} and Theorem \ref{thm:affbdd}.
Since the forgetful functor $\Dcal(B_{\square}) \to \Dcal(A_{\square})$ is conservative and preserves all small limits, so it is enough to show that $$-\otimes_{A_{\square}}^{\Lbb} B_{\square} \colon \Dcal(A_{\square}) \to \Dcal(A_{\square})$$ preserves all small limits.
This follows from the description $$-\otimes_{A_{\square}}^{\Lbb} B_{\square}\simeq R\intHom_A(N_{B/A},-) \colon \Dcal(A_{\square}) \to \Dcal(A_{\square}).$$ 
\end{proof}

\begin{lemma}\label{lem:repani}
The functor $F \colon \Dcal(A_{\square}) \xrightarrow{-\otimes_{A_{\square}}^{\Lbb}B_{\square}} \Dcal(B_{\square}) \xrightarrow{\Map(B,-)} \Ani$ is representable by some compact object $N_{B/A}\in \Dcal(A_{\square})$.
Moreover we have an equivalence $$N_{B/A}\otimes_{A_{\square}}^{\Lbb}\pi_0A_{\square} \simeq N_{\pi_0B/\pi_0A}.$$
\end{lemma}
\begin{proof}
By Lemma \ref{lem:lim} and \cite[Proposition 5.5.2.7]{HTT}, the functor $F$ is represented by some object $N \in \Dcal(A_{\square})$.
Since $F$ preserves filtered colimits, $N$ is a compact object of $\Dcal(A_{\square})$ and in particular it is bounded.
Let us prove $N\otimes_{A_{\square}}^{\Lbb}\pi_0A_{\square} \simeq N_{\pi_0B/\pi_0A}$.
The object $N\otimes_{A_{\square}}^{\Lbb}\pi_0A_{\square}$ represents a functor 
$$G \colon \Dcal(\pi_0A_{\square}) \to \Dcal(A_{\square}) \xrightarrow{F} \Ani.$$
This functor can be written as $\Dcal(\pi_0A_{\square}) \xrightarrow{-\otimes_{\pi_0A_{\square}}^{\Lbb} \pi_0B_{\square}} \Dcal(\pi_0B_{\square}) \xrightarrow{\Map(\pi_0B,-)} \Ani$, and it represented by $N_{\pi_0B/\pi_0A}\in \Dcal(\pi_0A_{\square})$, so we get $N\otimes_{A_{\square}}^{\Lbb}\pi_0A_{\square} \simeq N_{\pi_0B/\pi_0A}$.
\end{proof}

\begin{theorem}\label{existN}
Let $N_{B/A}\in \Dcal(A_{\square})$ be as in Lemma \ref{lem:repani}.
Then there exists an equivalence of functors from $\Dcal(A_{\square})$ to $\Dcal(A_{\square})$
$$-\otimes_{A_{\square}}^{\Lbb}B_{\square} \simeq R\intHom_A(N_{B/A},-).$$
\end{theorem}
\begin{proof}
We have an equivalence of anima 
\begin{align*}
\Map_{\Dcal(A_{\square})}(N_{B/A},N_{B/A}) &\simeq \Map_{\Dcal(B_{\square})}(B,N_{B/A}\otimes_{A_{\square}}^{\Lbb} B_{\square})\\
&\simeq \Map_{\Dcal(A_{\square})}(A,N_{B/A}\otimes_{A_{\square}}^{\Lbb} B_{\square}).
\end{align*}
Let $f \colon A \to N_{B/A}\otimes_{A_{\square}}^{\Lbb} B_{\square}$ be a morphism which corresponds to $\id \colon N_{B/A} \to N_{B/A}$.
By \cite[Lemma A.4.8]{Mann22}, $f$ induces a natural transformation of functors
$$\eta\colon R\intHom_A(N_{B/A},-) \Rightarrow -\otimes_{A_{\square}}^{\Lbb}B_{\square} \colon \Dcal(A_{\square}) \to \Dcal(A_{\square}).$$
We will prove $\eta$ is an equivalence.
Both functors preserve all small limits and small colimits by Lemma \ref{lem:lim} and the fact that $N_{B/A}$ is compact, so it is enough to show that for a static object $M\in \Dcal(A_{\square})^{\heart}$, $\eta_M \colon R\intHom_A(N_{B/A},M) \to M\otimes_{A_{\square}}^{\Lbb}B_{\square}$ is an equivalence.
Since $M$ is a $\pi_0A_{\square}$-module, $\eta_M$ is equivalent to $$R\intHom_{\pi_0A}(N_{B/A}\otimes_{A_{\square}}^{\Lbb} \pi_0A_{\square},M) \to M\otimes_{\pi_0A_{\square}}^{\Lbb}\pi_0B_{\square},$$ which is an equivalence by the fact that $N_{B/A}\otimes_{A_{\square}}^{\Lbb}\pi_0A_{\square} \simeq N_{\pi_0B/\pi_0A}$.
\end{proof}

Next, we will check the condition in Proposition \ref{prop:desbdd}.

\begin{theorem}\label{thm:unifbdd}
For every positive integer $r$, $N_{B/A}^{\otimes r}$ lies in $\Dcal(A_{\square})^{\leq 1}$ (resp. $\Dcal(A_{\square})^{\leq 2}$).
Moreover, the natural map $\varinjlim_{[m] \in \Delta_{s}^{\op}}N_{B/A}^{\otimes(m+1)} \to A$ is an equivalence.
\end{theorem}
\begin{proof}
First, we prove the former claim.
Since we have an equivalence $N_{B/A}^{\otimes r}\simeq N_{B^r/A}$, where $B^r$ is the $r$-fold derived tensor product of $B$ over $A$, so we may assume $r=1$.
We assume that there exists an integer $t>1$ (resp. $t>2$) such that $H^t(N_{B/A})\neq 0$.
Then the natural map $N_{B/A} \to \tau^{\geq t}N_{B/A}$ is not zero.
Therefore, $H^0(\tau^{\geq t}N_{B/A} \otimes_{A_{\square}}^{\Lbb} B_{\square}) \cong H^0(R\intHom_A(N_{B/A}, \tau^{\geq t}N_{B/A}))$ is not zero.
However the functor $-\otimes_{A_{\square}}^{\Lbb} B_{\square}$ has finite Tor-dimension $\leq 1$ (resp. $\leq 2$) by Lemma \ref{lem:lim} , so it is a contradiction.

Next, we prove the latter claim.
By the first claim, we find that $\varinjlim_{[m] \in \Delta_{s}^{\op}}N_{B/A}^{\otimes(m+1)}$ is bounded above. 
Therefore, it is enough to show that $\varinjlim_{[m] \in \Delta_{s}^{\op}}N_{B/A}^{\otimes(m+1)}\otimes_{A_{\square}}^{\Lbb} \pi_0A \to \pi_0A$ is an equivalence by Lemma \ref{lem:bddeq}.
Since we have an equivalence $N_{B/A}\otimes_{A_{\square}}^{\Lbb} \pi_0A \simeq N_{\pi_0B/\pi_0A}$ by Lemma \ref{lem:repani},
it is enough to show that
\begin{align}
\varinjlim_{[m] \in \Delta_{s}^{\op}}N_{\pi_0B/\pi_0A}^{\otimes(m+1)} \to \pi_0A
\label{4.14}
\end{align}
is an equivalence, where $N_{\pi_0B/\pi_0A}^{\otimes(m+1)}$ is the $(m+1)$-fold derived tensor product of $N_{\pi_0B/\pi_0A}$ in $\Dcal(\pi_0A_{\square})$.
When $A$ is discrete, it follows from \cite[Theorem 2.15]{Mikami22}.
We assume that $\pi_0A$ is an affinoid $K$-algebra. By \cite[Theorem 4.13]{Mikami22} and Lemma \ref{lem:desiff}, $\pi_0A_{\square} \to \pi_0B_{\square}$ is descendable, so Lemma \ref{lem:desconv} shows that (\ref{4.14}) is an equivalence. 
\end{proof}

From the above, we can get the main theorem.
\begin{theorem}\label{thm:main}
The morphism $A_{\square}\to B_{\square}$ is descendable of index $\leq 2$ (resp. $\leq 3$).
\end{theorem}

By using \cite[Proposition 2.6.3, Proposition 2.9.6]{Mann22}, and Proposition \ref{prop:afftensor}, we can get the following corollaries.
\begin{corollary}
Let $f \colon A\to B$ be a faithfully flat morphism of discrete animated rings such that $\pi_0f \colon \pi_0A \to \pi_0B$ is finitely presented.
Let $B^{n/A}$ denote the $n$-fold derived tensor product of $B$ over $A$.
Then we have an equivalence of $\infty$-categories 
$$\Dcal(A_{\square}) \overset{\sim}{\lra} \varprojlim_{[n] \in \Delta} \Dcal((B^{(n+1)/A})_{\square}). $$
\end{corollary}

\begin{corollary}
Let $f \colon A\to B$ be a faithfully flat morphism of animated affinoid $K$-algebras.
Let $B^{n/A}$ denote the $n$-fold derived tensor product of $B$ over $(A,\Zbb)_{\square}$.
Then we have an equivalence of $\infty$-categories 
$$\Dcal(A_{\square}) \overset{\sim}{\lra} \varprojlim_{[n] \in \Delta} \Dcal((B^{(n+1)/A})_{\square}). $$
\end{corollary}
\begin{proof}
It follows from Theorem \ref{thm:main}, Lemma \ref{lem:flattening}, and Corollary \ref{cor:ratcov}.
\end{proof}


\end{document}